\documentclass[11pt]{article}
\usepackage{amsmath,amssymb,amsthm,amscd}

\newtheorem{theorem}{Theorem}[section]
\newtheorem{lemma}[theorem]{Lemma}
\newtheorem{proposition}[theorem]{Proposition}
\newtheorem{corollary}[theorem]{Corollary}
\newtheorem{example}[theorem]{Example}

\theoremstyle{plain}

\theoremstyle{definition}
\newtheorem{definition}[theorem]{Definition}
\newtheorem{remark}[theorem]{Remark}

\numberwithin{equation}{section}

\renewcommand{\labelenumi}{\textup{(\theenumi)}}

\renewcommand{\phi}{\varphi}

\newcommand{\Pic}{\operatorname{Pic}}
\newcommand{\Aut}{\operatorname{Aut}}
\newcommand{\Out}{\operatorname{Out}}
\newcommand{\Int}{\operatorname{Int}}

\newcommand{\Homeo}{\operatorname{Homeo}}

\newcommand{\diag}{\operatorname{diag}}
\newcommand{\id}{\operatorname{id}}
\newcommand{\Ker}{\operatorname{Ker}}

\newcommand{\Ad}{\operatorname{Ad}}

\newcommand{\K}{\mathcal{K}}
\newcommand{\C}{\mathcal{C}}

\newcommand{\A}{\mathcal{A}}
\newcommand{\B}{\mathcal{B}}
\newcommand{\D}{\mathcal{D}}

\newcommand{\N}{\mathbb{N}}

\newcommand{\Z}{\mathbb{Z}}
\newcommand{\bbC}{\mathbb{C}}
\newcommand{\Zp}{{\mathbb{Z}}_+}

\def\OA{{{\mathcal{O}}_A}}
\def\ON{{{\mathcal{O}}_N}}
\def\DA{{{\mathcal{D}}_A}}

\def\OB{{{\mathcal{O}}_B}}
\def\DB{{{\mathcal{D}}_B}}
\def\OTA{{{\mathcal{O}}_{\tilde{A}}}}
\def\DTA{{{\mathcal{D}}_{\tilde{A}}}}

\title{Relative Morita equivalence of Cuntz--Krieger algebras and 
flow equivalence of topological Markov shifts}
\author{Kengo Matsumoto \\
Department of Mathematics \\
Joetsu University of Education \\
Joetsu, 943-8512, Japan
}
\begin{document}
\maketitle

\date{ }

\def\det{{{\operatorname{det}}}}

\begin{abstract}
In this paper, 
we will introduce notions of  relative version of imprimitivity bimodules
and relative version of strong Morita equivalence for pairs of 
$C^*$-algebras $(\mathcal{A}, \mathcal{D})$ such that $\mathcal{D}$ is a $C^*$-subalgebra of $\mathcal{A}$ with certain conditions.
We will then prove that two pairs
$(\mathcal{A}_1, \mathcal{D}_1)$ and
$(\mathcal{A}_2, \mathcal{D}_2)$
are relatively  Morita equivalent if and only if 
their  relative stabilizations are isomorphic.
In particularly, for two pairs 
$(\mathcal{O}_A, \mathcal{D}_A)$
and
$(\mathcal{O}_B, \mathcal{D}_B)$
of Cuntz--Krieger algebras with their canonical masas,
they are relatively Morita equivalent if and only if 
their underlying two-sided topological Markov shifts
$(\bar{X}_A,\bar{\sigma}_A)$ and
$(\bar{X}_B,\bar{\sigma}_B)$
are flow equivalent.
We also introduce a relative version of the Picard group
${\operatorname{Pic}}(\mathcal{A}, \mathcal{D})$
for the pair $(\mathcal{A}, \mathcal{D})$ of $C^*$-algebras
and study them for the Cuntz--Krieger pair $(\mathcal{O}_A, \mathcal{D}_A)$.
\end{abstract}




\def\OA{{{\mathcal{O}}_A}}
\def\OB{{{\mathcal{O}}_B}}
\def\OZ{{{\mathcal{O}}_Z}}
\def\OTA{{{\mathcal{O}}_{\tilde{A}}}}
\def\SOA{{{\mathcal{O}}_A}\otimes{\mathcal{K}}}
\def\SOB{{{\mathcal{O}}_B}\otimes{\mathcal{K}}}
\def\SOZ{{{\mathcal{O}}_Z}\otimes{\mathcal{K}}}
\def\SOTA{{{\mathcal{O}}_{\tilde{A}}\otimes{\mathcal{K}}}}
\def\DA{{{\mathcal{D}}_A}}
\def\DB{{{\mathcal{D}}_B}}
\def\DZ{{{\mathcal{D}}_Z}}
\def\DTA{{{\mathcal{D}}_{\tilde{A}}}}
\def\SDA{{{\mathcal{D}}_A}\otimes{\mathcal{C}}}
\def\SDB{{{\mathcal{D}}_B}\otimes{\mathcal{C}}}
\def\SDZ{{{\mathcal{D}}_Z}\otimes{\mathcal{C}}}
\def\SDTA{{{\mathcal{D}}_{\tilde{A}}\otimes{\mathcal{C}}}}
\def\BC{{{\mathcal{B}}_C}}
\def\BD{{{\mathcal{B}}_D}}
\def\OAG{{\mathcal{O}}_{A^G}}
\def\OBG{{\mathcal{O}}_{B^G}}
\def\Max{{{\operatorname{Max}}}}
\def\Per{{{\operatorname{Per}}}}
\def\PerB{{{\operatorname{PerB}}}}
\def\Homeo{{{\operatorname{Homeo}}}}
\def\HA{{{\frak H}_A}}
\def\HB{{{\frak H}_B}}
\def\HSA{{H_{\sigma_A}(X_A)}}
\def\Out{{{\operatorname{Out}}}}
\def\Aut{{{\operatorname{Aut}}}}
\def\Ad{{{\operatorname{Ad}}}}
\def\Inn{{{\operatorname{Inn}}}}
\def\det{{{\operatorname{det}}}}
\def\exp{{{\operatorname{exp}}}}
\def\cobdy{{{\operatorname{cobdy}}}}
\def\Ker{{{\operatorname{Ker}}}}
\def\ind{{{\operatorname{ind}}}}
\def\id{{{\operatorname{id}}}}
\def\supp{{{\operatorname{supp}}}}
\def\co{{{\operatorname{co}}}}
\def\Sco{{{\operatorname{Sco}}}}
\def\U{{{\mathcal{U}}}}
\bigskip

{\bf Contents}

\medskip

1. Introduction

2. Relative $\sigma$-unital $C^*$-algebras.

3. Relative imprimitivity bimodules and relative Morita equivalence.

4. Isomorphism of relative stabilizations.

5. Relative full corners.

6. Relative Morita equivalence in Cuntz--Krieger pairs.

7. Corner isomorphisms in Cuntz--Krieger pairs.

8. Relative Picard groups.

9. Relative Picard groups of Cuntz--Krieger pairs.

10. Appendix:  Picard groups of Cuntz--Krieger algebras.

\section{Introduction}
In \cite{Rieffel1}, M. Rieffel introduced the notion of imprimitivity bimodule
for $C^*$-algebras as a Hilbert $C^*$-bimodule satisfying certain conditions 
from a viewpoint of representation theory of groups, 
so that he defined the notion of strong Morita equivalence in $C^*$-algebras.
Let $\A$ and $\B$ be  $C^*$-algebras.
An $\A$--$\B$-bimodule $X$
means a Hilbert $C^*$-bimodule  
with  a left $\A$-module structure and an $\A$-valued inner product 
${}_\A\!\langle\hspace{3mm}\mid\hspace{3mm}\rangle$
and 
with a right $\B$-module structure and a $\B$-valued inner product 
$\langle\hspace{3mm}\mid\hspace{3mm}\rangle_\B$
satisfying some comparability conditions
(see \cite{Paschke}, \cite{Rieffel1}, \cite{KW}, \cite{RW}, etc.). 
It is said to be full if the ideals 
spanned by 
$\{{}_\A\!\langle x \mid y \rangle \mid x,y \in X\}$
and
$\{ \langle x \mid y \rangle_\B \mid x,y \in X\}$
are dense in $\A$ and in $\B$, respectively. 
 If a full $\A$--$\B$-bimodule $X$
 further satisfies
the condition
\begin{equation*}
 {}_\A\!\langle x \mid y \rangle z
=
x \langle y \mid z \rangle_\B
\quad
\text{ for } x, y, z \in X,
\end{equation*}
it is called an $\A$--$\B$-imprimitivity bimodule.
Two $C^*$-algebras 
$\A$ and $\B$ are said to be strong Morita equivalent if there exists an 
$\A$--$\B$-imprimitivity bimodule, 
which means that $\A$ and $\B$ have same representation theory.
 Brwon--Green--Rieffel in \cite{BGR}
have shown that two $\sigma$-unital $C^*$-algebras
$\A$ and $\B$ are strong Morita equivalent if and only if they are stably isomorphic,
that is 
$\A\otimes \K$ is isomorphic to    
$\B\otimes \K$,
where 
$\K$ denotes the $C^*$-algebra of compact operators on a separable infinite dimensional Hilbert space. 

In this paper, we will study Morita equivalence of $C^*$-algebras from a view point of symbolic dynamical systems.
For an irreducible non-permutation matrix $A =[A(i,j)]_{i,j=1}^N$
with entries in $\{0,1\}$,
two-sided topological Markov shift
$(\bar{X}_A,\bar{\sigma}_A)$ are defined 
as a topological dynamical system
on the shift space
$\bar{X}_A$
consisting of  two-sided
sequences 
$(x_n)_{n \in \Z}$ 
of
$x_n \in \{1,\dots,N\}$ 
such that $A(x_n,x_{n+1}) = 1$ for all $n \in \Z$ 
with the shift homeomorphism
$\bar{\sigma}_{A}((x_n)_{n \in {\Z}})=(x_{n+1} )_{n \in \Z}$
on the compact Hausdorff space 
$\bar{X}_A$.
J. Cuntz and W. Krieger introduced a $C^*$-algebra 
$\OA$ associated to the matrix $A$
(\cite{CK}).
The $C^*$-algebra is called the Cuntz--Krieger algebra,
which is a universal unique $C^*$-algebra generated by
partial isometries $S_1,\dots,S_N$
subject to the relations:
\begin{equation} 
\sum_{j=1}^N S_j S_j^* = 1, 
\qquad
S_i^* S_i = \sum_{j=1}^N A(i,j) S_jS_j^*, 
\quad i=1,\dots,N. \label{eq:CK}
\end{equation} 
Since the stable isomorphism class of $\OA$
does not have complete informations about the underlying dynamical system
$(\bar{X}_A,\bar{\sigma}_A)$, 
we need some extra structure to $\OA$
to study $(\bar{X}_A,\bar{\sigma}_A)$. 
In this paper,  we  consider the pair $(\OA,\DA)$
where $\DA$ is the $C^*$-subalgebra of $\OA$
generated by the projections of the form:
$S_{i_1}\cdots S_{i_n}S_{i_n}^* \cdots S_{i_1}^*,
i_1,\dots,i_n =1,\dots,N$.
We call the pair $(\OA,\DA)$ the Cuntz--Krieger pair.
As in \cite{MaPacific}, the isomorphism  class of the pair $(\OA,\DA)$
is a complete invariant of the continuous orbit equivalence class
of the underlying one-sided topological Markov shift
$(X_A,\sigma_A)$.
As one of the remarkable relationships between symbolic dynamics and Cuntz--Krieger algebras,
Cuntz--Krieger showed in \cite{CK} that
if topological Markov shifts
$(\bar{X}_A,\bar{\sigma}_A)$ and
$(\bar{X}_B,\bar{\sigma}_B)$
are flow equivalent,
 then there exists an isomorphism
$\Phi:\SOA\longrightarrow \SOB$ such that 
$\Phi(\SDA) = \SDB$,
where $\mathcal{C}$ denotes the maximal commutative $C^*$-subalgebra
of $\K$ consisting of the diagonal elements.
Recently H. Matui and the author have proved that
the converse implication also holds, so that 
$(\bar{X}_A,\bar{\sigma}_A)$ and
$(\bar{X}_B,\bar{\sigma}_B)$
are flow equivalent if and only if there exists an isomorphism
$\Phi:\SOA\longrightarrow \SOB$ such that 
$\Phi(\SDA) = \SDB$ (\cite{MMKyoto}).
We call the pair $(\SOA,\SDA)$ the stabilized Cuntz--Krieger pair
or the relative stabilization of $(\OA,\DA)$,
so that the isomorphism class of the relative stabilization of $(\OA,\DA)$
is a complete invariant for the flow equivalence class of the underlying two-sided topological Markov shift $(\bar{X}_A,\bar{\sigma}_A)$.

In this paper, 
we will introduce notions of  relative version of imprimitivity bimodules
and of relative version of strong Morita equivalence for pairs of 
$C^*$-algebras $(\mathcal{A}, \mathcal{D})$ such that 
$\mathcal{D}$ is a $C^*$-subalgebra of $\mathcal{A}$ 
for which $\mathcal{D}$ has an orthogonal countable approximate unit for $\A$.
Such a pair is said to be relative $\sigma$-unital.
If $\D$ contains the unit of $\A$, the pair is relative $\sigma$-unital.  
Relative version of strong Morita equivalence is called the relative Morita equivalence.
We will first show the following theorem for relative $\sigma$-unital pair $(\A, \D)$ 
of $C^*$-algebras:

\begin{theorem}[{Lemma \ref{lem:stab}, Theorem \ref{thm:RMEPHI2}} and Theorem \ref{thm:fullcorners}]
Let  
$(\A_1, \D_1)$
and 
$(\A_2, \D_2)$
be relative $\sigma$-unital pairs of $C^*$-algebras.
Then the following assertions are mutually equivalent:
\begin{enumerate}
\item 
$(\A_1, \D_1)$ and
$
(\A_2, \D_2)$ are relatively Morita equivalent.
\item 
$(\A_1\otimes\K, \D_1\otimes\C)$ and 
$(\A_2\otimes\K, \D_2\otimes\C)$ are relatively Morita equivalent.
\item 
There exists an isomorphism 
$\Phi: \A_1\otimes\K \longrightarrow \A_2\otimes \K$ of $C^*$-algebras
such that 
$\Phi(\D_1\otimes\C)=\D_2\otimes \C$.
\item
$(\A_1, \D_1)$
and 
$(\A_2, \D_2)$
are complementary relative full corners.
\end{enumerate}
\end{theorem} 
We will second apply the above theorem to the Cuntz--Krieger 
pair $(\OA,\DA)$
 and clarify relationships between relative Morita equivalence  
and flow equivalence of underlying topological dynamical systems.

\begin{theorem}[{Theorem \ref{thm:FERME} and Theorem \ref{thm:sumCK}, cf. {\cite[Corollary 3.8]{MMKyoto}}}] 
Let $A, B$ be irreducible non-permutation matrices with entries in $\{0,1\}$.  
Let
$(\OA,\DA), (\OB,\DB)$ be the associated Cuntz--Krieger pairs.
Then the following assertions are mutually equivalent:
\begin{enumerate}
\item 
$(\OA, \DA)$ and
$
(\OB, \DB)
$ are relatively Morita equivalent.
\item 
$(\OA\otimes\K, \DA\otimes\C)
$ and
$
(\OB\otimes\K, \DB\otimes\C)
$ are relatively Morita equivalent.
\item 
There exists an isomorphism 
$\Phi: \OA\otimes\K \longrightarrow \OB\otimes \K$ of $C^*$-algebras
such that 
$\Phi(\DA\otimes\C)=\DB\otimes \C$.
\item
$(\OA, \DA)$
and
$
(\OB, \DB)
$
are corner isomorphic.
\item
The two-sided topological Markov shifts 
$(\bar{X}_A, \bar{\sigma}_A)$ and 
$(\bar{X}_B, \bar{\sigma}_B)$ 
are flow equivalent.
\end{enumerate}
\end{theorem} 
By using J. Franks' s Theorem \cite{Franks} (cf. \cite{BF}, \cite{PS}),
the last assertion (5) is equivalent to the following (6):

\medskip

(6) \,\, The groups $\Z^N/{(\id - A)\Z^N}$ and $\Z^M/{(\id - B)\Z^M}$
are isomorphic and $\det(\id -A) = \det(\id -B)$,

\medskip

where $N$  is the size of the matrix $A$ and
 $M$ is that of $B$.
Hence we know that 
the group 
$\Z^N/{(\id - A)\Z^N}$ with the value $\det(\id -A) $
is a complete invariant of the relative Morita equivalence class
of the Cuntz--Krieger pair
$(\OA,\DA)$.

In \cite{BGR}, Brown--Green--Rieffel 
introduced the notion of the Picard group $\Pic(\A)$
for a $C^*$-algebra to study equivalence classes of imprimitivity  bimodules of  
$C^*$-algebras.
Natural isomorphism classes $[X]$
of  imprimitivity bimodules $X$  over $\A$
form a group
under  the relative tensor product
$[X]\cdot [Y] =[X\otimes_\A Y]$.
The group is called the Picard group for the
$C^*$-algebra $\A$ and is written 
$\Pic(\A)$, that are considered as a sort of generalizations of 
automorphism group $\Aut(\A)$ of $\A$.
We will introduce relative version of the Picard group
$\Pic(\A,\D)$ as the group of 
$(\A,\D)$--$(\A,\D)$-relative imprimitivity bimodules 
and study their structure for the Cuntz--Krieger pairs
$(\OA,\DA)$.
Let
\begin{equation*}
\Aut_\circ(\OA,\DA) 
= \{\alpha\in \Aut(\OA)\mid \alpha(\DA) = \DA, \alpha_* = \id \text{ on } K_0(\OA)\}.
\end{equation*}
Its  quotient group
 $\Aut_\circ(\OA,\DA)/\Int(\OA,\DA)$ by 
$\Int(\OA,\DA)$
is denoted by
$\Out_\circ(\OA,\DA)$.
Let 
$
\Aut_1( \Z^N/{(\id - A^t)\Z^N} )
$ 
be a subgroup of 
$\Aut( \Z^N/{(\id - A^t)\Z^N})$ defined by
$$
\Aut_1( \Z^N/{(\id - A^t)\Z^N})
= \{ \xi \in \Aut( \Z^N/{(\id - A^t)\Z^N}) \mid \xi([1]) = [1] \}
$$
where 
$[1] \in  \Z^N/{(\id - A^t)\Z^N}$
denotes the class of the vector 
$(1,\dots,1) $ in $\Z^N$. 
It is well-known that there exists an isomorphism
$\epsilon_A:K_0(\OA)\longrightarrow \Z^N/{(\id - A^t)\Z^N}$
such that $\epsilon([1_{\OA}]) =[1]$ (\cite{Cu3}).
We will obtain the following structure theorem for $\Pic(\OA,\DA)$.

\begin{theorem}[{Theorem \ref{thm:PicOADA} and Theorem \ref{thm:PicOADA2}}]
Let $A$ be an irreducible non-permutation matrix.
Then there exist short exact sequences:
\begin{gather*}
1\longrightarrow \Out_\circ(\OA,\DA) 
\overset{\bar{\Psi}}{\longrightarrow}\Pic(\OA,\DA) 
\overset{K_*}{\longrightarrow}\Aut( \Z^N/{(\id - A^t)\Z^N})
\longrightarrow 1, 
\\
1 \longrightarrow \Out(\OA,\DA) 
\overset{\bar{\Psi}}{\longrightarrow}\Pic(\OA,\DA) 
 \overset{K_*}{\longrightarrow}\Aut( \Z^N/{(\id - A^t)\Z^N}) 
/ \Aut_1( \Z^N/{(\id - A^t)\Z^N})
\longrightarrow 1. 
\end{gather*}
\end{theorem}
In Appendix of the paper, we refer to the ordinary Picard groups
$\Pic(\OA)$ for Cuntz--Krieger algebras $\OA$,
and especially the ordinary Picard groups  
$\Pic({\mathcal{O}}_N)$ for Cuntz algebras ${\mathcal{O}}_N$
(Theorem \ref{thm:PicON} and Corollary \ref{cor:PicONprime}).

\section{Relative $\sigma$-unital $C^*$-algebras}
For a $C^*$-algebra $\A$, we denote by 
$M(\A)$ its multiplier $C^*$-algebra (cf. \cite{WO}).
The locally convex topology on $M(\A)$ generated by the seminorms
$x \longrightarrow \|x a\|, \,x \longrightarrow \|a x \|$ 
for $a \in \A$ is called the strict topology.
Throughout the paper, we denote by 
$\{e_{ij}\}_{i,j\in \N}$ the matrix units on the separable infinite dimensional Hilbert space
$\ell^2(\N)$.
The $C^*$-algebra generated by them is denoted by $\K$
which is the $C^*$-algebra of all compact operators on $\ell^2(\N)$.
The $C^*$-subalgebra of $\K$ generated by diagonal projections
$\{e_{i,i}\}_{i \in \N}$ is denoted by $\C$.

A $C^*$-algebra is said to be $\sigma$-unital if it has a countable approximate unit.
We will first introduce a notion of relative version of a $\sigma$-unital $C^*$-algebra.
\begin{definition}
A pair $(\A,\D)$ of $C^*$-algebras $\A, \D$
is called {\it relative}\/ $\sigma$-{\it unital}\/
if it satisfies the following conditions:
\begin{enumerate}
\item $\D$ is a $C^*$-subalgebra of $\A$.
\item $\D$ contains a countable approximate unit for $\A$.
\item There exists a sequence $a_n \in \A, n=1,2,\dots $ of elements such that  
{\begin{enumerate}
\renewcommand{\theenumi}{\roman{enumi}}
\renewcommand{\labelenumi}{\textup{(\theenumi)}}
\item $a_n^* d a_n, \, a_n d a_n^* \in \D$ for all $d \in \D$ and  $n =1,2,\dots$.
\item $\sum_{n=1}^{\infty}a_n^* a_n =1$ in the strict topology of $M(\A)$.
\item $a_n d a_m^* = 0$ 
for all $d \in \D$ and $n,m \in \N$ with $n\ne m$.
\end{enumerate}}
\end{enumerate} 
\end{definition}
We call the sequence $\{a_n\}_{n\in \N}$ 
satisfying the three conditions (a), (b), (c) 
{\it a relative approximate unit}\/ for the pair $(\A,\D)$.

\begin{remark}
By the above condition (2), we know that 
$M(\D)$ is a $C^*$-subalgebra of $M(\A)$ in natural way
(cf. \cite[p 46, 2G]{WO}).  
\end{remark}

\begin{lemma}
Assume that 
$(\A,\D)$ is a relative $\sigma$-unital pair of $C^*$-algebras.
Let $\{a_n\}_{n\in \N}$ be a relative approximate unit for $(\A,\D)$.
Then we have
\begin{enumerate}
\renewcommand{\theenumi}{\roman{enumi}}
\renewcommand{\labelenumi}{\textup{(\theenumi)}}
\item $a_n^*  a_n, \, a_n  a_n^* \in \D$ for all $n =1,2,\dots$.
\item $b_n =\sum_{k=1}^{n}a_k^* a_k$ belongs to $\D$ and
the sequence $\{b_n\}_{n\in \N}$ 
is a countable approximate unit for $\A$.
\end{enumerate}
\end{lemma}
\begin{proof}
(i)
Take and fix $k\in \N$.
Since $\sum_{n=1}^{\infty}a_n^* a_n =1$ in  $M(\A)$,
we have
$0\le a_k^* a_k \le 1$ so that $\| a_k\| \le 1$.
As $\D$ has an approximate unit for $\A$, for any $\epsilon>0$,
there exists $d \in \D$ such that 
$\| a_k - d a_k \| <\epsilon$,
so that
$\|a_k^* a_k -a_k^* d a_k \| <\epsilon$.
The condition $a_k^* d a_k \in \D$
ensures us that 
$a_k^* a_k$ belongs to $\D$.
Similarly we know that
$a_k a_k^*$ belongs to $\D$.
 
(ii)
Since 
$b_n =\sum_{k=1}^{n}a_k^* a_k$
converges to $1$ in the strict topology of $M(\A)$,
$\{b_n\}_{n\in \N}$ is an approximate unit for $\A$. 
\end{proof}

\begin{lemma}\label{lem:orthapprunit}
Let $\D$ be a $C^*$-subalgebra of $\A$.
Then $(\A,\D)$ is relative $\sigma$-unital if and only if 
there exists a sequence $d_n \in \D, n=1,2,\dots$ such that 
\begin{enumerate}
\renewcommand{\theenumi}{\alph{enumi}}
\item $d_n \ge 0, \, n =1,2,\dots$.
\item $\sum_{n=1}^{\infty} d_n =1$ in the strict topology of $M(\A)$.
\item $d_n d d_m^* = 0$ for all $d \in \D$ and $n,m \in \N$ with $n\ne m$.
\end{enumerate}
\end{lemma}
\begin{proof}
Suppose that $(\A,\D)$ is relative $\sigma$-unital.
Take a relative approximate unit $\{a_n\}_{n\in \N}$ in $\A$.
Put $d_n = a_n^* a_n$. By the preceding lemma, $d_n$ belongs to $\D$ 
and satisfies the desired properties.  
Conversely, suppose that there exists a sequence  $d_n$ in $\D$ satisfying 
the above three conditions.
Put $a_n = \sqrt{d_n}$, which becomes a relative approximate unit for $(\A,\D)$.
\end{proof}
We call the sequence $\{d_n\}_{n \in \N}$ in $\D$ satisfying 
the conditions
(a), (b), (c) in Lemma \ref{lem:orthapprunit}
{\it an orthogonal approximate unit for}\/ $(\A,\D)$.

\begin{example}

{\bf 1.} If a $C^*$-subalgebra $\D$ of $\A$  contains the unit of $\A$,
the pair $(\A,\D)$ is relative $\sigma$-unital
by putting $d_1 = 1$ and $d_n = 0$ for $n=2,3,\dots $.

{\bf 2.} Let $\A =\K$
and
$\D =\C$.
Then  the pair $(\A,\D)$ is relative $\sigma$-unital 
by putting 
$d_n =e_{n,n}, n \in \N$
where $\{e_{n,m}\}_{n,m\in \N}$
is the matrix units of $\K$.  
\end{example}
More generally we know the following proposition.
\begin{proposition}
If $(\A,\D)$ is relative $\sigma$-unital, 
so is 
$(\A\otimes\K,\D\otimes\C)$.
\end{proposition}
\begin{proof}
Take an orthogonal approximate unit 
$\{d_n\}_{n\in \N}$ in $\D$ for
the pair $(\A,\D)$.
Put
$d_{(n,m)} = d_n\otimes e_{m,m}$ for $n,m=1,2,\dots$.
It is straightforward to see that 
the sequence
$d_{(n,m)}, n,m=1,2,\dots$
 becomes an orthogonal approximate unit for the pair
$(\A\otimes\K,\D\otimes\C)$.
\end{proof}
We call the pair $(\A\otimes\K,\D\otimes\C)$
{\it the relative stabilization}\/ for $(\A,\D)$. 
\begin{corollary}
If a $C^*$-subalgebra $\D$ of $\A$  contains the unit of $\A$,
both the pairs $(\A,\D)$ and $(\A\otimes\K,\D\otimes\C)$
are relative $\sigma$-unital. 
\end{corollary}

\section{Relative imprimitivity bimodules and relative Morita equivalence}
Let $(\A_1, \D_1)$ and $(\A_2, \D_2)$
be relative $\sigma$-unital pairs of $C^*$-algebras.
\begin{definition}\label{def:rib}
Let $X$ be an $\A_1$--$\A_2$-Hilbert $C^*$-bimodule. 
Put
\begin{equation*}
X_D =\{ x \in X \mid 
 {}_{\A_1}\!\langle x d_2 \mid x \rangle \in \D_1\text{ for all } d_2 \in \D_2,
\, \langle x \mid d_1 x \rangle_{\A_2} \in \D_2\text{ for all } d_1 \in \D_1\}.
\end{equation*}
The $\A_1$--$\A_2$-Hilbert $C^*$-bimodule $X$ is called an 
$(\A_1, \D_1)$--$(\A_2, \D_2)$-{\it relative imprimitivity bimodule}\/
if it satisfies the following conditions:
\begin{enumerate}
\item $X$ is an $\A_1$--$\A_2$-imprimitivity bimodule.
\item 
There exists a sequence $x_n \in X_D, n=1,2,\dots $ such that  
{\begin{enumerate}
\renewcommand{\theenumi}{\roman{enumi}}
\renewcommand{\labelenumi}{\textup{(\theenumi)}}
\item 
$\sum_{n=1}^{\infty} \langle x_n \mid x_n \rangle_{\A_2}=1$
 in the strict topology of $M(\A_2)$.
\item  
${}_{\A_1}\!\langle x_n d_2 \mid x_m \rangle =0$ 
for all $d_2 \in \D_2$ and $n,m \in \N$ with $n\ne m$.
\end{enumerate}}
\item 
There exists a sequence $y_n \in X_D, n=1,2,\dots $ such that  
{\begin{enumerate}
\renewcommand{\theenumi}{\roman{enumi}}
\renewcommand{\labelenumi}{\textup{(\theenumi)}}
\item 
$\sum_{n=1}^{\infty} {}_{\A_1}\!\langle y_n \mid y_n \rangle=1$
 in the strict topology of $M(\A_1)$.
\item  
$\langle y_n\mid  d_1 y_m \rangle_{\A_2} =0$ for all $d_1 \in \D_1$ and
 $n,m \in \N$ with  $n\ne m$.
\end{enumerate}}
\end{enumerate} 
\end{definition}
\begin{remark}
\begin{enumerate}
\item
Since $X$ is an $\A_1$--$\A_2$-imprimitivity bimodule,
norms on $X$ defined by their inner products coincide each other, 
that is,
$ \| {}_{\A_1}\!\langle x \mid x \rangle \|^{\frac{1}{2}}
           = \| \!\langle x \mid x \rangle_{\A_2} \|^{\frac{1}{2}} \|$
for $x \in X$ (cf. \cite[Proposition 3.1]{RW}). 
We denote the  norm by $\| x \|$. 
\item
The above elements $x_n,  y_n \in X_D$ in Definition \ref{def:rib}
satisfy the inequalities
\begin{equation}
{}_{\A_1}\!\langle x_n \mid x_n \rangle \le 1,
\qquad
\langle y_n \mid y_n \rangle_{\A_2} \le 1
\end{equation}
because of the inequality
$$
{}_{\A_1}\!\langle x_n \mid x_n \rangle 
\le 
\| {}_{\A_1}\!\langle x_n \mid x_n \rangle \| = \| \langle x_n \mid x_n \rangle_{\A_2} \|
\le
\| \sum_{n=1}^\infty \langle x_n \mid x_n \rangle_{\A_2} \| = 1
$$
and of a similar inequality for 
$
\langle y_n \mid y_n \rangle_{\A_2}$.
\item
 Both the left action of $\A_1$ and the right action of $\A_2$
on $X$ are non-degenerate, that is,
$\overline{\A_1 X} = X= \overline{X\A_2}$.
More strongly we see that 
$\overline{\D_1 X} = X= \overline{X\D_2}$. 
In fact, 
for $d_1\in \D_1$ and $x \in X$, the following inequalities hold
\begin{align*}
\| x - d_1 x \|^2
& =  \| {}_{\A_1}\!\langle x - d_1 x \mid  x - d_1 x \rangle \| \\
& =  \| {}_{\A_1}\!\langle x \mid  x \rangle 
  - d_1 {}_{\A_1}\!\langle x  \mid  x \rangle 
       - {}_{\A_1}\!\langle x  \mid  x \rangle  d_1^*
  + d_1 {}_{\A_1}\!\langle x  \mid  x \rangle  d_1^* \| \\
&\le \| {}_{\A_1}\!\langle x \mid  x \rangle 
  - d_1 {}_{\A_1}\!\langle x  \mid  x \rangle \| 
  +   \| {}_{\A_1}\!\langle x  \mid  x \rangle 
  - d_1 {}_{\A_1}\!\langle x  \mid  x \rangle  \|
\| d_1^*\|. 
\end{align*} 
As $\D_1$ has a countable approximate unit for $\A_1$,
we have a sequence $d_1(n)$ in $\D_1$ such that 
$\lim_{n\to\infty}\| x - d_1(n) x\| = 0$ so that
$\overline{\D_1 X} = X$.
\end{enumerate}
\end{remark}
\begin{lemma}\label{lem:xxD}
For $x \in X_D$ we have
\begin{enumerate}
\renewcommand{\theenumi}{\roman{enumi}}
\renewcommand{\labelenumi}{\textup{(\theenumi)}}
\item 
$ {}_{\A_1}\!\langle x \mid x \rangle \in \D_1$. 
\item  
$\langle x \mid x \rangle_{\A_2} \in \D_2$. 
\end{enumerate}
\end{lemma}
\begin{proof}
(i)
Let $x \in X_D$.
For $d_2 \in \D_2$, we have
\begin{equation}
\langle x -x d_2 \mid x - x d_2 \rangle_{\A_2}
=\langle x \mid x \rangle_{\A_2} - \langle x \mid x \rangle_{\A_2}d_2 
 - d_2^* \langle x \mid x \rangle_{\A_2} + d_2^* \langle x \mid x \rangle_{\A_2} d_2.
\label{eq:xxd2} 
\end{equation}
Now $\D_2$ contains an approximate unit for $\A_2$, 
the equality \eqref{eq:xxd2} shows that
for any $\epsilon>0$ there exists an element $d_2 \in \D_2$ such that
$\| \langle x -x d_2 \mid x - x d_2 \rangle_{\A_2} \| <\epsilon.$
Since $X$ is an $\A_1$--$\A_2$-imprimitivity bimodule, we see that
$\| {}_{\A_1}\!\langle x -x d_2 \mid x - x d_2 \rangle \| <\epsilon$
by \cite[Lemma 3.11]{RW}.
By the Cauchy--Schwartz inequality (cf. \cite[Lemma 2.5]{RW}),
we have
\begin{align*}
\| {}_{\A_1}\!\langle x -x d_2 \mid x  \rangle \|^2
=& \| {}_{\A_1}\!\langle x -x d_2 \mid x  \rangle^*  {}_{\A_1}\!\langle x -x d_2 \mid x  \rangle \| \\
\le &  \| {}_{\A_1}\!\langle x -x d_2 \mid x -xd_2 \rangle\|
 \| {}_{\A_1}\!\langle x  \mid x  \rangle \| \\
< & \epsilon  \| {}_{\A_1}\!\langle x  \mid x  \rangle \|.
\end{align*}
Hence we have
\begin{equation}
\|  {}_{\A_1}\!\langle x  \mid x  \rangle -  {}_{\A_1}\!\langle x d_2 \mid x  \rangle \|^2
=
\|  {}_{\A_1}\!\langle x -x d_2 \mid x  \rangle  \|^2
<  \epsilon  \| {}_{\A_1}\!\langle x  \mid x  \rangle \|. \label{eq:xxxd2x}
\end{equation}
As
$  {}_{\A_1}\!\langle x d_2 \mid x  \rangle$ belongs to $\D_1$,
we conclude that 
$ {}_{\A_1}\!\langle x  \mid x  \rangle$ belongs to $\D_1$.

(ii) is similarly shown to (i).
\end{proof}
\begin{lemma}\hspace{6cm}
\begin{enumerate}
\renewcommand{\theenumi}{\roman{enumi}}
\renewcommand{\labelenumi}{\textup{(\theenumi)}}
\item 
We have
$
z = \sum_{n=1}^\infty {}_{\A_1}\!\langle z \mid x_n \rangle x_n
$
for $z \in X$
which converges in the norm of $X$, and 
${}_{\A_1}\!\langle x_n \mid x_m \rangle =0$
for $n,m \in \N$ with $n\ne m$.
\item
We have
$
z= \sum_{n=1}^\infty y_n \langle y_n \mid z \rangle_{\A_2}
$
for $z \in X$
which converges in the norm of $X$, and 
$\langle y_n \mid y_m \rangle =0$
for $n,m \in \N$ with $n\ne m$.
\end{enumerate}
\end{lemma}
\begin{proof}
(i)
As $X= \overline{X\D_2}$,
for $z \in X$ and $\epsilon>0$
there exists $d_2 \in \D_2$ such that 
$\| z - z d_2\|<\epsilon$.
Since 
$\sum_{n=1}^\infty \langle x_n \mid x_n\rangle_{\A_2} =1$
in the strict topology of $M(\A_2)$, we may find $K\in \N$
such that 
$\| \sum_{n=1}^K d_2 \langle x_n \mid x_n\rangle_{\A_2} -d_2 \| <\epsilon$.
Therefore we have
\begin{align*}
  & \| z - \sum_{n=1}^K {}_{\A_1}\!\langle z \mid x_n\rangle x_n \| \\
=& \| z - \sum_{n=1}^K z \langle x_n \mid x_n\rangle_{\A_2} \| \\
\le &\| z -  z d_2 \|
+ \| z d_2 - \sum_{n=1}^K z d_2 \langle x_n \mid x_n\rangle_{\A_2} \| 
+ \| \sum_{n=1}^K z d_2 \langle x_n \mid x_n\rangle_{\A_2} 
 -    \sum_{n=1}^K z \langle x_n \mid x_n\rangle_{\A_2} \| \\
\le & \| z -  z d_2 \|
+ \| z \| \| d_2 - \sum_{n=1}^K d_2 \langle x_n \mid x_n\rangle_{\A_2} \| 
+ \| ( z d_2 - z) \sum_{n=1}^K \langle x_n \mid x_n\rangle_{\A_2}  \| \\
=& (2 + \| z \|) \epsilon,
\end{align*}
so that 
$\sum_{n=1}^\infty {}_{\A_1}\langle z \mid x_n \rangle x_n$
converges to $z$ in the norm of $X$.

As in the proof of Lemma \ref{lem:xxD}, for $n,m\in \N$ with $n\ne m$,
there exists $d_2(k)\in \D_2$
such that 
\begin{equation*}
\lim_{n\to\infty}
\|  {}_{\A_1}\!\langle x_n  \mid x_m  \rangle 
 -  {}_{\A_1}\!\langle x_n d_2 \mid x_m  \rangle \|^2
= 0.
\end{equation*}
Since
$
{}_{\A_1}\!\langle x_n d_2 \mid x_m  \rangle 
=0,
$
we have
$
{}_{\A_1}\!\langle x_n \mid x_m  \rangle 
=0.
$
 \end{proof}
The sequences $\{ x_n\}_{n \in \N}, \{ y_n\}_{n \in \N} \subset X_D$
 satisfying the conditions (2), (3) 
in Definition \ref{def:rib}
are called
{\it a relative left basis}\/, 
{\it a relative right basis},\/ respectively. 
The pair 
$(\{ x_n\}, \{ y_n\})$ is called {\it a relative basis}\/ for $X$.

We arrive at our definition of relative version of strong Morita equivalence.
\begin{definition}
Two relative $\sigma$-unital pairs of $C^*$-algebras
$(\A_1, \D_1)$ and $(\A_2, \D_2)$ are said to be
{\it relatively Morita equivalent}\/ if there exists an
$(\A_1, \D_1)$--$(\A_2, \D_2)$-relative imprimitivity bimodule.
In this case we write
$
(\A_1, \D_1)\underset{\operatorname{RME}}{\sim}(\A_2, \D_2).
$
\end{definition}
\begin{lemma}\label{lem:theta}
Let $(\A_1, \D_1)$ and $(\A_2, \D_2)$ be
relative $\sigma$-unital pairs of $C^*$-algebras.
If there exists an isomorphism
$\theta:\A_1\longrightarrow \A_2$ of $C^*$-algebras
such that
$\theta(\D_1)= \D_2,$
then we have
$
(\A_1, \D_1)\underset{\operatorname{RME}}{\sim}(\A_2, \D_2).
$
In particular, 
for   a relative $\sigma$-unital pair $(\A,\D)$ of $C^*$-algebras,
we have
$
(\A, \D)\underset{\operatorname{RME}}{\sim}(\A, \D).
$
\end{lemma}
\begin{proof}
Let $a_n \in \A_1, n\in \N$ be a relative approximate unit for $(\A_1,\D_1)$.
Put
$X_\theta = \A_1$ 
as vector space having module structure  and inner products given by
\begin{gather*}
a_1\cdot x \cdot a_2 := a_1 x \theta^{-1}(a_2) \quad
\text{ for } a_1 \in \A_1, \, a_2 \in \A_2, \, x \in X_\theta,\\ 
{}_{\A_1}\!\langle x  \mid y  \rangle = xy^*, \qquad 
\langle x  \mid y  \rangle_{\A_2} = \theta(x^*y)
\quad \text{ for } x,y \in X_\theta.
\end{gather*}
Put $x_n = a_n, n\in \N.$ 
We have for $d_1 \in \D_1, d_2 \in \D_2$
\begin{equation*} 
{}_{\A_1}\!\langle x_n d_2 \mid x_n  \rangle = a_n \theta^{-1}(d_2) a_n^* \in \D_1, 
\qquad
\langle x_n \mid d_1 x_n  \rangle_{\A_2} = \theta(a_n^* d_1 a_n) \in \D_2
\end{equation*}
so that $x_n \in {(X_\theta)}_D$. 
We also have
\begin{gather*}
\sum_{n=1}^{\infty} \langle x_n \mid x_n \rangle_{\A_2}
=
\sum_{n=1}^{\infty} \theta(a_n^* a_n) = 1, \\
\intertext{and}
{}_{\A_1}\!\langle x_n d_2 \mid x_m \rangle = a_n \theta^{-1}(d_2) a_m^* =0
\quad
\text{ for all } d_2 \in \D_2 \text{ and } n, m \in \N \text{ with } n\ne m.
\end{gather*}
Similarly by putting $y_n = a_n^*$
we have
\begin{equation*} 
{}_{\A_1}\!\langle y_n d_2 \mid y_n  \rangle = a_n^* \theta^{-1}(d_2) a_n \in \D_1, 
\qquad
\langle y_n \mid d_1 y_n  \rangle_{\A_2} = \theta(a_n d_1 a_n^*) \in \D_2
\end{equation*}
so that $y_n \in {(X_\theta)}_D$. 
We also have
\begin{gather*}
\sum_{n=1}^{\infty} {}_{\A_1}\!\langle y_n \mid y_n \rangle
=
\sum_{n=1}^{\infty} a_n^* a_n = 1, \\
\intertext{and}
\langle y_n  \mid d_1 y_m \rangle = \theta(a_n d_1 a_m^* )=0
\quad
\text{ for all } d_1 \in \D_1 \text{ and } n,m \in \N \text{ with } n\ne m.
\end{gather*}
Hence $(\{x_n\}, \{y_n\})$
is a relative basis for $X_\theta$
so that  $X_\theta$ 
becomes an  $(\A_1,\D_1)$--$(\A_2,\D_2)$-relative imprimitivity bimodule to show
 $(\A_1, \D_1)\underset{\operatorname{RME}}{\sim}(\A_2, \D_2).$
\end{proof}

We will next show that the relation 
$\underset{\operatorname{RME}}{\sim}$ is an equivalence relation in 
relative $\sigma$-unital pairs of $C^*$-algebras.
\begin{lemma}\label{lem:trans}
Suppose that
$X_{12}$ is an  $(\A_1, \D_1)$--$(\A_2, \D_2)$-relative imprimitivity bimodule
and
$X_{23}$ is an  $(\A_2, \D_2)$--$(\A_3, \D_3)$-relative imprimitivity bimodule.
Then the relative tensor product
$X_{12}\otimes_{\A_2}X_{23}$ of bimodules
is an  $(\A_1, \D_1)$--$(\A_3, \D_3)$-relative imprimitivity bimodule.
\end{lemma}
\begin{proof}
Take relative bases 
$(\{x_n\}, \{ y_n\})$ for $X_{12}$ 
and 
$(\{z_n\}, \{w_n\})$ for $X_{23}$.
We will show that the pair
$(\{ x_n \otimes z_m\}_{n,m}, \{ y_n \otimes w_m\}_{n,m})$ 
becomes a relative basis
for  
$X_{12}\otimes_{\A_2}X_{23}$. 
For $d_3 \in \D_3, d_1 \in \D_1$, we have
\begin{align*}
{}_{\A_1}\!\langle (x_n \otimes z_m) d_3 \mid x_n \otimes z_m \rangle
=& {}_{\A_1}\!\langle x_n \otimes (z_m d_3) \mid x_n \otimes z_m  \rangle
= {}_{\A_1}\!\langle x_n {}_{\A_2}\!\langle z_m d_3 \mid z_m \rangle \mid x_n \rangle, \\
\langle x_n \otimes z_m \mid d_1( x_n \otimes z_m )\rangle_{\A_3}
=& 
\langle x_n \otimes z_m \mid (d_1 x_n) \otimes z_m \rangle_{\A_3}
=\langle z_m \mid \langle x_n  \mid d_1 x_n \rangle_{\A_2} z_m \rangle_{\A_3}.
\end{align*}
As
$ {}_{\A_2}\!\langle z_m d_3 \mid z_m \rangle \in \D_2$, 
we have 
${}_{\A_1}\!\langle x_n {}_{\A_2}\!\langle z_m d_3 \mid z_m \rangle \mid x_n \rangle \in \D_1$
so that 
${}_{\A_1}\!\langle (x_n \otimes z_m) d_3 \mid x_n \otimes z_m \rangle
\in \D_1$.
Similarly we know that 
$ {}_{\A_1}\!\langle x_n {}_{\A_2}\!\langle z_m d_3 \mid z_m \rangle \mid x_n \rangle
\in \D_3$.

We also have
\begin{align*}
\sum_{n,m=1}^\infty 
\langle x_n \otimes z_m \mid  x_n \otimes z_m \rangle_{\A_3}
=&
\sum_{n,m=1}^\infty 
\langle z_m \mid \langle x_n  \mid  x_n \rangle_{\A_2} z_m \rangle_{\A_3} \\
=&
\sum_{m=1}^\infty 
\langle z_m \mid (\sum_{n=1}^\infty \langle x_n  \mid  x_n \rangle_{\A_2}) z_m \rangle_{\A_3} \\
= &
\sum_{m=1}^\infty 
\langle z_m \mid z_m \rangle_{\A_3}
=1.
\end{align*}

For $d_3 \in \D_3$, we have 
\begin{equation*}
{}_{\A_1}\!\langle (x_n \otimes z_m) d_3 \mid x_l \otimes z_k \rangle
= {}_{\A_1}\!\langle x_n \otimes (z_m d_3) \mid x_l \otimes z_k  \rangle
= {}_{\A_1}\!\langle x_n {}_{\A_2}\!\langle z_m d_3 \mid z_k \rangle \mid x_l \rangle.
\end{equation*}
If $m\ne k$, then ${}_{\A_2}\!\langle z_m d_3 \mid z_k \rangle =0$.
If $n\ne l$, then 
$ {}_{\A_1}\!\langle x_n {}_{\A_2}\!\langle z_m d_3 \mid z_k \rangle \mid x_l \rangle =0$
because
${}_{\A_2}\!\langle z_m d_3 \mid z_k \rangle\in \D_2$.
Hence 
if $(n,m) \ne (l,k)$, we have 
$
{}_{\A_1}\!\langle (x_n \otimes z_m) d_3 \mid x_l \otimes z_k \rangle
=0
$
and know that  the sequence 
$\{ x_n \otimes z_m\}_{n,m}$
is a relative left basis for $X_{12}\otimes_{\A_2}X_{23}$. 
By a similar argument, 
we know that  
$\{ y_n \otimes w_m\}_{n,m}$
is a relative right basis for $X_{12}\otimes_{\A_2}X_{23}$,
so that  
$(\{ x_n \otimes z_m\}_{n,m}, \{ y_n \otimes w_m\}_{n,m})$
is a relative  basis for 
$X_{12}\otimes_{\A_2}X_{23}.$
\end{proof}
Therefore we have
\begin{proposition}
Relative Morita equivalence
$\underset{\operatorname{RME}}{\sim}$ is an equivalence relation in 
relative $\sigma$-unital pairs of $C^*$-algebras.
\end{proposition}
\begin{proof}
The refrexisive law follows from Lemma \ref{lem:theta}.
We will show the symmetric law.
Suppose that 
$
(\A_1, \D_1)\underset{\operatorname{RME}}{\sim}(\A_2, \D_2)$
via relative imprimitivity bimodule $X_{12}$.
Then its conjugate module $\bar{X}_{12}$ denoted by $X_{21}$
becomes an
$(\A_2, \D_2)$--$(\A_1, \D_1)$-relative imprimitivity bimodule
(see \cite[Definition 6.17]{Rieffel1},
cf. \cite[p. 3443]{KW}),
so that 
$
(\A_1, \D_2)\underset{\operatorname{RME}}{\sim}(\A_1, \D_1)$.
The transitive law follows from Lemma \ref{lem:trans}.
\end{proof}
\begin{lemma}\label{lem:stab}
Let $(\A,\D)$ be a relative $\sigma$-unital pair of $C^*$-algebras.
Then we have
$$
(\A, \D)\underset{\operatorname{RME}}{\sim}(\A\otimes\K, \D\otimes\C).
$$
\end{lemma}
\begin{proof}
Let $a_n \in \A, n\in \N$
be a relative approximate unit for $(\A,\D)$.
Let $\{e_{n,m}\}_{n,m\in \N}$ be the matrix units of $\K$.
Define $X = \A \otimes e_{1,1}\K$.
By identifying $\A$ with $\A\otimes{\mathbb{C}}e_{1,1}$,
 $X$ has a natural structure of 
$\A$--$\A\otimes\K$-imprimitivity bimodule. 
Put $x_{n,m} = a_n \otimes e_{1,m} \in X, n,m \in \N$.
For $d_1 \in \D$ and $d_2=d\otimes f \in \D\otimes\C$, 
we have
\begin{align*}
{}_{\A}\!\langle x_{n,m} d_2 \mid x_{n,m} \rangle
= & (a_n \otimes e_{1,m})(d\otimes f)(a_n \otimes e_{1,m})^*
= a_n d a_n^*\otimes e_{1,m} f e_{1,m}^* \in \D\otimes {\bbC}e_{1,1}, \\
\langle x_{n,m} \mid d_1 x_{n,m} \rangle_{\A\otimes\K}
= & (a_n \otimes e_{1,m})^*(d\otimes e_{1,1})(a_n \otimes e_{1,m})
= a_n^* d a_n\otimes e_{m,1}e_{1,1} e_{1,m} \in  \D\otimes\C,
\end{align*}
so that 
$x_{n,m}$ belongs to $X_D$
under the identification between 
$\D$ with $\D\otimes{\mathbb{C}}e_{1,1}$.
We also have
\begin{equation*}
\sum_{n,m=1}^\infty 
\langle x_{n,m} \mid  x_{n,m} \rangle_{\A\otimes\K}
=
\sum_{n,m=1}^\infty 
(a_n \otimes e_{1,m})^* (a_n \otimes e_{1,m})
=
\sum_{n,m=1}^\infty 
a_n^* a_n   \otimes e_{1,m}^* e_{1,m} = 1 \otimes 1
\end{equation*}
in $M(\A\otimes\K)$.
For $d_2=d\otimes f \in \D\otimes\C$, 
we have
\begin{equation*}
{}_{\A}\!\langle x_{n,m} d_2 \mid x_{k,l} \rangle
=  (a_n \otimes e_{1,m})(d\otimes f)(a_k \otimes e_{1,l})^*
= a_n d a_k^*\otimes e_{1,m} f e_{1,l}^*. 
\end{equation*}
If $n\ne k$, we have $a_n d a_k^* =0$.
If $m\ne l$, we have $e_{1,m} f e_{1,l}^*=0$.
Hence 
if $(n,m) \ne (k,l)$, we have 
$
{}_{\A}\!\langle x_{n,m} d_2 \mid x_{k,l} \rangle =0$.

Put $y_n = a_n^* \otimes e_{1,1}$.
Then 
for $d_1 \in \D$ and $d_2=d\otimes f \in \D\otimes\C$, 
we have
\begin{align*}
{}_{\A}\!\langle y_n d_2 \mid y_n \rangle
= & (a_n^* \otimes e_{1,1})(d\otimes f)(a_n^* \otimes e_{1,1})^*
= a_n^* d a_n\otimes e_{1,1} f e_{1,1}^* \in \D\otimes {\bbC}e_{1,1}, \\
\langle y_n \mid d_1 y_n \rangle_{\A\otimes\K}
= & (a_n^* \otimes e_{1,1})^*(d\otimes e_{1,1})(a_n^* \otimes e_{1,1})
= a_n d a_n^*\otimes e_{1,1}  \in  \D\otimes\C,
\end{align*}
so that 
$y_n$ belongs to $X_D$.
We also have 
\begin{gather*}
\sum_{n=1}^\infty 
{}_{\A}\!\langle y_n  \mid y_n \rangle
 =
\sum_{n=1}^\infty 
a_n^*a_n \otimes e_{1,1}
= 1 \otimes e_{1,1},\\
\intertext{and}
\langle y_n \mid d_1 y_m \rangle
 = (a_n^* \otimes e_{1,1})^*(d\otimes e_{1,1})(a_m^* \otimes e_{1,1})
  = a_n d a_m^*  \otimes e_{1,1}
  = 0 \text{ for } n\ne m.  
\end{gather*}
Therefore 
$X$ becomes an  $(\A,\D)$--$(\A\otimes\K,\D\otimes\C)$-relative imprimitivity bimodule,
so that 
$
(\A, \D)\underset{\operatorname{RME}}{\sim}(\A\otimes\K, \D\otimes\C).
$
\end{proof}
\begin{example}
For  $m,k\in \N$,
let 
$\A_1= M_m(\bbC), \D_1 = {\diag}(M_m(\bbC)) = {\bbC}^m,$
and
$\A_2= M_k(\bbC), \D_2 = {\diag}(M_k(\bbC)) = {\bbC}^k$.
Then we have
$
(\A_1, \D_1)\underset{\operatorname{RME}}{\sim}(\A_2, \D_2).
$
\end{example}
We will present an 
$(\A_1, \D_1)$--$(\A_2, \D_2)$-relative imprimitivity bimodule in the followung way.
Let $\A_0, \D_0$ be 
$M_{m+k}(\bbC), \diag(M_{m+k}(\bbC))$, respectively.
Let $p_1, p_2 $ be the projections in $\D_0$
defined by
$$
p_1=(\overbrace{1,\cdots,1}^m,\overbrace{0,\cdots,0}^k),
\quad
p_2=(\overbrace{0,\cdots,0}^m,\overbrace{1,\cdots,1}^k).
$$
We then have
$$
\A_1 = p_1\A_0 p_1, \qquad\D_1 = \D_0 p_1
\quad\text{ and }
\quad
\A_2 = p_2\A_0 p_2, \qquad\D_1 = \D_0 p_2.
$$
Put
$X =p_1\A_0 p_2
$
with natural $\A_1$--$\A_2$-bimodule structure
and inner products such taht
$$
 {}_{\A}\!\langle x \mid y \rangle = x y^*, \qquad
 \langle x \mid y \rangle_{\A_2} = x^* y
 \quad
 \text{ for }
 x,y \in X.
$$
It is not difficult to see that 
$X$ becomes an $(\A_1, \D_1)$--$(\A_2, \D_2)$-relative imprimitivity bimodule
so that 
$
(\A_1, \D_1)\underset{\operatorname{RME}}{\sim}(\A_2, \D_2).
$

\section{Isomorphism of relative stabilizations}
In this section, we devote to proving the following theorem, 
which is a relative version of a part of Brown--Green--Rieffel Theorem
\cite[Theorem 1.2]{BGR}. 
\begin{theorem}\label{RMEPHI}
Suppose 
$(\A_1, \D_1)\underset{\operatorname{RME}}{\sim}(\A_2, \D_2).
$
Then there exists an isomorphism 
$\Phi: \A_1\otimes\K \longrightarrow \A_2\otimes \K$ of $C^*$-algebras
such that 
$\Phi(\D_1\otimes\C)=\D_2\otimes \C$.
\end{theorem} 
Suppose that $X$ is an
$(\A_1, \D_1)$--$(\A_2, \D_2)$-relative imprimitivity bimodule.
Let $\bar{X}$ be the conjugate bimodule of $X$
(\cite[Definition 6.17]{Rieffel1}, cf. \cite[p.3443]{KW}).
The corresponding element in $\bar{X}$ to $y \in X$
is denoted by $\bar{y}$.
It is straightforward to see that 
$\bar{X}$ is 
$(\A_2, \D_2)$--$(\A_1, \D_1)$-relative imprimitivity bimodule.
We define the relative linking pair $(\A_0, \D_0)$
 by setting
\begin{align}
\A_0 
& = \{
\begin{bmatrix}
a_1      &  x  \\
\bar{y} & a_2
\end{bmatrix}
\mid 
a_1 \in \A_1, a_2 \in \A_2, x \in X, \bar{y} \in \bar{X} \}, \label{eq:linkinga}\\
\D_0 
& = \{
\begin{bmatrix}
d_1      &  0  \\
0        & d_2
\end{bmatrix}
\mid 
d_1 \in \D_1, d_2 \in \D_2 \}. \label{eq:linkingd}
\end{align}
As in \cite[p.350]{BGR}, the products between two elements of $\A_0$ is defined by
\begin{equation*}
\begin{bmatrix}
a_1      &  x  \\
\bar{y} & a_2
\end{bmatrix}
\begin{bmatrix}
b_1      &  z  \\
\bar{w} & b_2
\end{bmatrix}
:=
\begin{bmatrix}
a_1 b_1 +  {}_{\A_1}\!\langle x \mid w \rangle  &  a_1 z + x b_2  \\
\bar{y} b_1 + a_2 \bar{w} & \langle y \mid z \rangle_{\A_2} + a_2 b_2
\end{bmatrix}.
\end{equation*}
Let $X\oplus \A_2$
be the Hilbert $C^*$-right module over $\A_2$
with the natural right action of $\A_2$ 
and
$\A_2$-valued right inner product defined by
\begin{equation*}
\langle
\begin{bmatrix}
 x \\
a_2
\end{bmatrix}
 \mid 
\begin{bmatrix}
y \\
b_2
\end{bmatrix}
 \rangle_{\A_2} 
:= \langle x  \mid y  \rangle_{\A_2} +  a_2b_2. 
\end{equation*}
The algebra $\A_0$ acts on $X\oplus \A_2$ by 
\begin{equation*}
\begin{bmatrix}
a_1      &  x  \\
\bar{y} & a_2
\end{bmatrix}
\begin{bmatrix}
z \\
b_2
\end{bmatrix}
=
\begin{bmatrix}
a_1 z +  x b_2 \\
\langle y\mid z \rangle_{\A_2} + a_2 b_2
\end{bmatrix}.
\end{equation*}
As seen in \cite[Lemma 3.20]{RW},  
$\A_0$ itself is a $C^*$-subalgebra of all bounded adjointable operators on
the Hilbert $C^*$-right module $X\oplus \A_2$.
We set
\begin{equation}
P_1 
 =
\begin{bmatrix}
1    &  0  \\
0    & 0
\end{bmatrix},
\qquad
P_2 
 =
\begin{bmatrix}
0    & 0  \\
0    & 1
\end{bmatrix}. \label{eq:cornerprojection}
\end{equation}
They satisfy $P_1 + P_2 = 1$ and
\begin{equation}
P_1 \A_0 P_1 = \A_1, \qquad \D_0 P_1 =\D_1\quad \text{ and }
P_2 \A_0 P_2 = \A_2, \qquad \D_0 P_2 =\D_2. \label{eq:APDPcorner}
\end{equation}
To prove Theorem \ref{RMEPHI},
we provide several lemmas.
\begin{lemma}\label{lem:UNTN}
Let $(\{x_n\}, \{y_n\}) \subset X_D$ 
be a relative bases for $X$.
\begin{enumerate}
\renewcommand{\theenumi}{\roman{enumi}}
\renewcommand{\labelenumi}{\textup{(\theenumi)}}
\item 
Put 
$
U_n =
\begin{bmatrix}
0    & x_n  \\
0    & 0
\end{bmatrix}
\in\A_0, n \in \N.
$
The sequence $U_n$ satisfies the following conditions:
{\begin{enumerate}
\renewcommand{\theenumi}{\alph{enumi}}
\item $P_2 =\sum_{n=1}^\infty U_n^* U_n $ 
which converges in the strict topology of $M(\A_0)$.
\item $U_n U_n^* \le P_1$ and $U_n U_m^* =0$ for $n\ne m$. 
\item $U_n \D_0 U_n^* \subset \D_0 P_1 = \D_1$.
\item $U_n^* \D_0 U_n \subset \D_0 P_2 = \D_2$.
\end{enumerate}}
\item
Put 
$
T_n =
\begin{bmatrix}
0    & 0  \\
\bar{y}_n    & 0
\end{bmatrix}
\in\A_0, n \in \N.
$
The sequence $T_n$ satisfies the following conditions:
{\begin{enumerate}
\renewcommand{\theenumi}{\alph{enumi}}
\item $P_1 =\sum_{n=1}^\infty T_n^* T_n $
which  converges in the strict topology of $M(\A_0)$.
\item $T_n T_n^* \le P_2$ and $T_n T_m^* =0$ for $n\ne m$. 
\item $T_n \D_0 T_n^* \subset \D_0 P_2 = \D_2$.
\item $T_n^* \D_0 T_n \subset \D_0 P_1 = \D_1$.
\end{enumerate}}
\end{enumerate}
\end{lemma}
\begin{proof}
(i)
For $d_1 \in \D_1, d_2 \in \D_2,$
we have
\begin{equation}
U_n^*
\begin{bmatrix}
d_1 & 0  \\
0    & d_2
\end{bmatrix}
U_n 
=
\begin{bmatrix}
0    & 0  \\
0    & \langle x_n \mid d_1 x_n \rangle_{\A_2}
\end{bmatrix}. 
\label{eq:undun}
\end{equation}
Since $x_n \in X_D$ and $d_1 \in \D_1$,
we have
$\langle x_n \mid d_1 x_n \rangle_{\A_2} \in \D_2$,
so that
$U_n^* \D_0 U_n \subset \D_0 P_2$,
which shows (d).
Since we have
\begin{equation}
U_n^*U_n 
=
\begin{bmatrix}
0    & 0  \\
0    & \langle x_n \mid x_n \rangle_{\A_2}
\end{bmatrix}
\end{equation}
the equality
$\sum_{n=1}^\infty \langle x_n \mid x_n \rangle_{\A_2} =1$
implies
$\sum_{n=1}^\infty U_n^* U_n = P_2$
which shows (a).
And also
for $d_1 \in \D_1, d_2 \in \D_2,$
we have
\begin{equation}
U_n
\begin{bmatrix}
d_1 & 0  \\
0    & d_2
\end{bmatrix}
U_n^* 
=
\begin{bmatrix}
{}_{\A_1}\!\langle x_n d_2 \mid x_n \rangle    & 0  \\
0    & 0
\end{bmatrix}. 
\label{eq:undun2}
\end{equation}
Since $x_n \in X_D$ and $d_2 \in \D_2$,
we have
$
{}_{\A_1}\!\langle x_n d_2 \mid x_n \rangle \in \D_1
$
so that
$U_n \D_0 U_n^* \subset \D_0 P_1$,
which shows (c).
Since we have
\begin{equation}
U_n U_m^* 
=
\begin{bmatrix}
{}_{\A_1}\!\langle x_n  \mid x_m \rangle    & 0  \\
0    &  0
\end{bmatrix}. 
\end{equation}
the inequality
$ {}_{\A_1}\!\langle x_n \mid x_n \rangle \le 1$
implies
$U_n U_n^* \le P_1$
and 
${}_{\A_1}\!\langle x_n  \mid x_m \rangle =0$ 
for $n,m\in \N$ with $n\ne m$.
which shows (b).
\end{proof}

\begin{lemma}
The pair $(\A_0, \D_0)$ is relative $\sigma$-unital.
\end{lemma}
\begin{proof}
Keep the above notations.
Put
$a_n = 
\begin{bmatrix}
0   & x_n  \\
\bar{y}_n  &  0
\end{bmatrix}
= U_n + T_n.
$
It then follows that 
$$
\sum_{n=1}^\infty a_n^* a_n
=\sum_{n=1}^\infty U_n^* U_n + \sum_{n=1}^\infty T_n^* T_n 
= P_2 + P_1= 1. 
$$
For $d_1 \in \D_1, d_2 \in \D_2,$
we have
\begin{align*}
a_n^*
\begin{bmatrix}
d_1 & 0  \\
0    & d_2
\end{bmatrix}
a_n 
& =
U_n^*
\begin{bmatrix}
d_1 & 0  \\
0    & d_2
\end{bmatrix}
U_n 
+
T_n^*
\begin{bmatrix}
d_1 & 0  \\
0    & d_2
\end{bmatrix}
T_n \\ 
& =
\begin{bmatrix}
\langle \bar{y}_n \mid d_2 \bar{y}_n \rangle_{\A_1} & 0 \\
0 & \langle x_n \mid d_1 x_n \rangle_{\A_2}   
\end{bmatrix}
\in \D_1 \oplus \D_2 =\D_0.
\end{align*}
Similarly 
we have
$a_n
\begin{bmatrix}
d_1 & 0  \\
0    & d_2
\end{bmatrix}
a_n^*
\in \D_1 \oplus \D_2.
$ 
We also have
$a_n d a_m^* = (U_n + T_n) d (U_m + T_m)^* = U_n d U_m^* + T_n d T_m^* =0
$
for $d = d_1 + d_2 \in \D_1\oplus \D_2$
and $n\ne m$. 
Hence $\{a_n\}$ is a relative approximate unit for $(\A_0,\D_0)$
to show $(\A_0,\D_0)$ is relative $\sigma$-unital.
\end{proof}

Let us decompose the set 
 $\N$ of natural numbers
into disjoint infinite subsets $\N = \cup_{j=1}^{\infty} {\N}_j$,
and decompose $\N_j$ for each $j$ 
once again into disjoint infinite sets
$\N_j = \cup_{k=0}^{\infty} {\N}_{j_k}.$
Let $\{ e_{i,j} \}_{i,j \in \N}$ 
be the set of matrix units which generate the algebra 
${\K} = \K(\ell^2(\N)).$
Put the projections
$f_j = \sum_{i\in {\N}_j} e_{i,i}$ 
and
$f_{j_k} = \sum_{i\in {\N}_{j_k}} e_{i,i}.$
Take a partial isometry
$s_{j_k,j}$ such that
$
s_{j_k,j}^*s_{j_k,j} = f_j,
s_{j_k,j}s_{j_k,j}^* = f_{j_k}
$
and put $s_{j,j_k} =s_{j_k,j}^*$.
We set for $n=1,2,\dots, $
\begin{align*}
u_n = \sum_{k=1}^{\infty} U_k \otimes s_{n_k,n}, 
& \qquad w_n = P_1 \otimes s_{{n_0},n} +  u_n,\\ 
t_n = \sum_{l=1}^{\infty} T_l \otimes s_{n_l,n},
& \qquad z_n =  P_2 \otimes s_{{n_0},n}  + t_n. 
\end{align*}
Then we have
\begin{lemma}[{cf. \cite[Lmma 3.3]{MaPre2016}}] \label{lem:3.3}
For each $n\in \N$, we have
\begin{enumerate}
\renewcommand{\theenumi}{\roman{enumi}}
\renewcommand{\labelenumi}{\textup{(\theenumi)}}
\item $w_n$ is a partial isometry in $M(\A_0\otimes \K)$ such that 
{\begin{enumerate}
\renewcommand{\theenumi}{\alph{enumi}}
\item $w_n^* w_n = 1 \otimes f_n$.
\item $w_n w_n^* \le P_1\otimes f_n$. 
\item $w_n (\D_0 \otimes\C)w_n^* \subset \D_1\otimes\C$.
\item $w_n^* (\D_0\otimes\C) w_n \subset \D_2\otimes\C$.
\end{enumerate}}
\item  $z_n$ is a partial isometry in $M(\A_0\otimes \K)$ such that 
{\begin{enumerate}
\renewcommand{\theenumi}{\alph{enumi}}
\item $z_n^* z_n = 1 \otimes f_n$.
\item $z_n z_n^* \le P_2\otimes f_n$. 
\item $z_n (\D_0 \otimes\C)z_n^* \subset \D_2\otimes\C$.
\item $z_n^* (\D_0\otimes\C) z_n \subset \D_1\otimes\C$.
\end{enumerate}}
\end{enumerate}
\end{lemma}
\begin{proof}
(i)
Since $u_n^* u_n = P_2 \otimes f_n$, we have
$$
w_n^* w_n = P_1 \otimes f_{n} + u_n^* u_n = P_1 \otimes f_n + P_2 \otimes f_n 
                = 1 \otimes f_n. 
$$
As  
$u_n(P_2 \otimes s_{n,n_0}) =(P_2 \otimes s_{n,n_0}) u_n^* =0$,
we have 
$$
w_n w_n^*
 = P_1 \otimes f_{n_0} + u_n u_n^* 
 = P_1 \otimes f_{n_0} +
\sum_{k=1}^{\infty} U_k U_k^* \otimes f_{n_k}. 
$$
Since $f_{n_0}, f_{n_k} \le f_n$,
we have 
$$
w_n w_n^* \le P_1 \otimes f_n.
$$
(ii) is similarly shown to (i).
\end{proof}
We will construct and study the unitary $V_1$ in $M(\A_0\otimes\K)$ such that
$\Ad(V_1): \A_0\otimes\K \longrightarrow \A_1\otimes\K$
and
$\Ad(V_1)(\D_0\otimes\C)=\D_1\otimes\C$

Let $f_{n,m}$ be a partial isometry satisfying
$f_{n,m}^*f_{n,m} = f_m,\, f_{n,m}f_{n,m}^* = f_n.$
The following lemma is straightforward.
\begin{lemma} [{cf. \cite[Lemma 3.4]{MaPre2016}}] \label{lem:3.4}
We put
\begin{align*}
v_1 &  = w_1 = P_1 \otimes s_{1_0,1} + u_1, \\
v_{2n} & = (P_1 \otimes f_n - v_{2n-1}v_{2n-1}^*)(P_1 \otimes f_{n,n+1})
\quad \text{ for } 1\le n \in \N, \\
v_{2n-1} & = w_n(1\otimes f_n - v_{2n-2}^*v_{2n-2}) \quad \text{ for } 2\le n \in \N. 
\end{align*} 
Then we have
for $n\in \N$ 
\begin{enumerate}
\renewcommand{\theenumi}{\alph{enumi}}
\renewcommand{\labelenumi}{\textup{(\theenumi)}}
\item$v_{2n-2}^* v_{2n-2} + v_{2n-1}^* v_{2n-1} = 1\otimes f_n$.
\item$v_{2n-1} v_{2n-1}^* + v_{2n} v_{2n}^* = P_1\otimes f_n$.
\item $v_n(\D_0 \otimes \C)v_n^* \subset \D_1\otimes\C.$
\item $v_n^*(\D_1 \otimes \C)v_n \subset \D_0\otimes\C.$
\end{enumerate}
\end{lemma}
By the above lemma, one may see that 
the summation
$\sum_{n=1}^\infty v_n$
 converges in $M(\A_0\otimes\K)$ 
to certain partial isometry written $V_1$
 in the strict topology of $M(\A_0\otimes \K)$.
Similarly we obtain a partial isometry
$V_2 $  in $M(\A_0\otimes\K)$
constructed from
the preceding sequences 
$t_n,  z_n $ of partial isometries.
As a consequence, we obtain the following proposition. 
\begin{proposition}\label{prop:V}
Assume that
$(\A_1,\D_1)
\underset{\operatorname{RME}}{\sim}
(\A_2\D_2)$.
Let $(\A_0,\D_0)$ be the relative linking pair defined in 
\eqref{eq:linkinga}, \eqref{eq:linkingd}.
\begin{enumerate}
\renewcommand{\theenumi}{\roman{enumi}}
\renewcommand{\labelenumi}{\textup{(\theenumi)}}
\item There exists an isometry $V_1$ in $M(\A_0\otimes \K)$
such that 
{\begin{enumerate}
\item$V_1^* V_1 = 1\otimes 1$.
\item$V_1 V_1^* = P_1\otimes 1.$
\item$V_1(\D_0\otimes\C)V_1^* = \D_1\otimes\C.$
\item$V_1^*(\D_1\otimes\C)V_1 = \D_0\otimes\C.$
\end{enumerate}}
\item There exists an isometry $V_2$ in $M(\A_0\otimes \K)$
such that 
{\begin{enumerate}
\item$V_2^* V_2 = 1\otimes 1$.
\item$V_2 V_2^* = P_2\otimes 1.$
\item$V_2(\D_0\otimes\C)V_2^* = \D_2\otimes\C.$
\item$V_2^*(\D_2\otimes\C)V_2 = \D_0\otimes\C.$
\end{enumerate}}
\end{enumerate}
\end{proposition}
Therefore we reach the following theorem
\begin{theorem}\label{thm:RMEPHI2}
Let  
$(\A_1, \D_1)$
and 
$(\A_2, \D_2)$
be relative $\sigma$-unital pairs of $C^*$-algebras.
Then 
$(\A_1, \D_1)\underset{\operatorname{RME}}{\sim}(\A_2, \D_2)
$
if and only if 
there exists an isomorphism 
$\Phi: \A_1\otimes\K \longrightarrow \A_2\otimes \K$ of $C^*$-algebras
such that 
$\Phi(\D_1\otimes\C)=\D_2\otimes \C$.
\end{theorem} 
\begin{proof}
Suppose 
$(\A_1, \D_1)\underset{\operatorname{RME}}{\sim}(\A_2, \D_2).
$
Take isometries 
$V_1, V_2$ in $M(\A_0\otimes\K)$
as in Proposition \ref{prop:V}.
Put $\Phi = \Ad(V_2 V_1^*)$
which gives rise to
an isomorphism 
$\Phi: \A_1\otimes\K \longrightarrow \A_2\otimes \K$ of $C^*$-algebras
such that 
$\Phi(\D_1\otimes\C)=\D_2\otimes \C$.

Converse implication comes from  Lemma \ref{lem:theta}
and Lemma \ref{lem:stab}.
\end{proof}

\section{Relative full corners}
It is well-known that 
two $C^*$-algebras are strong Morita equivalent if and only if they are complementary full corners of some $C^*$-algebra (\cite[Theorem 1.1]{BGR}).
In this section, we will study a relative version of this fact. 
\begin{definition}
For a relative $\sigma$-unital pair $(\A,\D)$
of $C^*$-algebra,
a projection $P \in M(\D)$ is said to be 
 {\it relative full}\/ for $(\A,\D)$
if it satisfies the following conditions
\begin{enumerate}
\item $P d = d P$ for all $d \in \D$.
\item There exists an sequence $a_n \in \A, n=1,2,\dots $ such that  
{\begin{enumerate}
\renewcommand{\theenumi}{\alph{enumi}}
\item $a_n^* d a_n\in \D, \, a_n d a_n^* \in \D P$ for all $d \in \D$ and  $n =1,2,\dots$.
\item $\sum_{n=1}^{\infty}a_n^* P a_n =1-P$ in the strict topology of $M(\A)$.
\item $a_n d a_m^* = 0$ 
for all $d \in \D$ and $n,m \in \N$ with $n\ne m$.
\end{enumerate}}
\end{enumerate} 
\end{definition}
We call the sequence $\{a_n\}_{n\in \N}$ 
satisfying the three conditions (a), (b), (c)
{\it a relative full sequence}\/ for $P$.
\begin{remark}
By the above condition (b), we know that 
$$
(\text{b}') \, \,\, a_n^* d P a_n \in \D(1-P) \quad \text{for all } d \in \D,\hspace{53mm}
$$
because we have
\begin{equation*}
(a_n^* d P a_n)^* a_n^* d P a_n
=a_n^*  P d^*  a_n a_n^* d P a_n 
\le \| d^*  a_n a_n^* d \| a_n^* P a_n \le 1 -P.
\end{equation*} 
\end{remark}
\begin{definition}\label{def:fullcorner}
Relative $\sigma$-unital pairs 
 $(\A_1, \D_1)$ and $(\A_2, \D_2)$ of $C^*$-algebras
are said to be {\it complementary relative full corners}\/ 
if there exists a relative $\sigma$-unital pair
$(\A_0,\D_0)$ of $C^*$-algebras
such that there exist relative full projections
$P_1, P_2 \in M(\D_0)$ such that 
\begin{equation}
P_1 + P_2 =1 \quad\text{ and }
\quad 
P_i \A_0 P_i = \A_i , \quad \D_0 P_i = \D_i,
\,\, \quad i=1,2.
\end{equation}
\end{definition}
\begin{proposition}\label{prop:fullcorners}
Let $(\A_1, \D_1)$ and $(\A_2, \D_2)$
be relative $\sigma$-unital pairs of $C^*$-algebras.
If they are complementary relative full corners,
then we have 
$
(\A_1, \D_1)\underset{\operatorname{RME}}{\sim}(\A_2, \D_2).
$
\end{proposition}
\begin{proof}
Let $(\A_0,\D_0)$ and $P_i \in M(\D_0), i=1,2$ 
be a relative $\sigma$-unital pair
 of $C^*$-algebras and projections, respectively, 
satisfying Definition \ref{def:fullcorner}.
Let $\{a_n\}$ and $\{b_n\}$ be relative full sequences for the projections
$P_1, P_2$, respectively. 
We set $X = P_1 \A_0 P_2$
and define two sequences by
$x_n = P_1 a_n P_2$ and 
$y_n = P_1 b_n^* P_2$.
For $d\in \D_0$, put
$d_i = d P_i, i=1,2.$
It then follows that 
\begin{align*}
{}_{\A_1}\!\langle x_n d_2 \mid x_n \rangle 
& = P_1 a_n P_2 d_2 P_2 a_n^* P_1 \in \D_0 P_1 =\D_1,\\
\langle x_n \mid d_1 x_n \rangle_{\A_2} 
& = P_2 a_n^* P_1 d_1 P_1 a_n P_2 \in \D_0 P_2 =\D_2.
\end{align*}
Hence $x_n$ belongs to $X_D$.
We also have 
\begin{align*}
\sum_{n=1}^\infty \langle x_n \mid x_n \rangle_{\A_2}
& = \sum_{n=1}^\infty P_2 a_n^* P_1 a_n P_2 = P_2,\\  
\intertext{and}
 {}_{\A_1}\!\langle x_n d_2 \mid x_m \rangle 
& =  P_1 a_n P_2 d P_2 a_m^* P_1 = 0 \quad \text{ for } n\ne m.
\end{align*}
Hence $\{ x_n\}$ is a relative left  basis for $X$.
Similarly we have 
\begin{align*}
{}_{\A_1}\!\langle y_n d_2 \mid y_n \rangle 
& = P_1 b_n^* P_2 d_2 P_2 b_n P_1 \in \D_0 P_1 =\D_1,\\
\langle y_n \mid d_1 y_n \rangle_{\A_2} 
& = P_2 b_n P_1 d_1 P_1 b_n^* P_2 \in \D_0 P_2 =\D_2.
\end{align*}
Hence $y_n$ belongs to $X_D$.
We also have 
\begin{align*}
\sum_{n=1}^\infty {}_{\A_1}\!\langle y_n \mid y_n \rangle
& = \sum_{n=1}^\infty P_1 b_n^* P_2 b_n P_1 = P_1,\\  
\intertext{and}
 \langle y_n \mid d_1 y_m \rangle_{\A_2} 
& =  P_2 b_n P_1 d P_1 b_m^* P_2 = 0 \quad \text{ for } n\ne m.
\end{align*}
Hence $\{ y_n\}$ is a relative right  basis for $X$.
 Therefore $X$ is an 
 $(\A_1, \D_1)$--$(\A_2, \D_2)$-relative imprimitivity bimodule,
so that we have 
$
(\A_1, \D_1)\underset{\operatorname{RME}}{\sim}(\A_2, \D_2).
$
\end{proof}
We obtain the following theorem.
\begin{theorem}\label{thm:fullcorners}
Let  
$(\A_1, \D_1)$
and 
$(\A_2, \D_2)$
be relative $\sigma$-unital pairs of $C^*$-algebras.
Then 
$(\A_1, \D_1)\underset{\operatorname{RME}}{\sim}(\A_2, \D_2)
$
if and only if 
$(\A_1, \D_1)$
and 
$(\A_2, \D_2)$
are complementary relative full corners.
\end{theorem} 
\begin{proof}
The if part has been proved in 
Proposition \ref{prop:fullcorners}.
To show the only if part,
suppose 
$(\A_1, \D_1)\underset{\operatorname{RME}}{\sim}(\A_2, \D_2).
$
Take $(\A_0, \D_0)$
the linking pair defined in 
\eqref{eq:linkinga}, \eqref{eq:linkingd}.
Let $P_1, P_2$ be the projections in $M(\D_0)$
defined by \eqref{eq:cornerprojection}.
Take the sequence $U_n, T_n$ as in Lemma \ref{lem:UNTN}.
The proof of Lemma \ref{lem:UNTN} shows us that 
the sequences  $a_n:=U_n$ and $b_n:=T_n$  
 are relative full sequences for $P_1$ and $P_2$, respectively,
so that $P_1$ and $P_2$ are relative full projections for $(\A_0,\D_0)$.
Since $P_1 + P_2 =1$,
the equalities \eqref{eq:APDPcorner}
show that 
$(\A_1, \D_1)$
and 
$(\A_2, \D_2)$
are complementary relative full corners.
\end{proof}

\section{Relative Morita equivalence in Cuntz--Krieger pairs}

In this section, we will study relative Morita equivalence particularly in 
Cuntz--Krieger algebras from a viewpoint of symbolic dynamical systems.
For a nonnegative matrix  $A =[A(i,j)]_{i,j=1}^N $,
 the associated directed graph  
$G_A = (V_A,E_A)$
consists of 
the vertex set
$V_A=\{v_1^A, \dots, v_N^A\}$
of $N$-vertices
and
 the edge set 
$E_A =\{ a_1, \dots, a_{N_A}\}$
where there are $A(i,j)$ edges from $v_i^A$ to $v_j^A$.
 For $a_i\in E_A$, denote by $t(a_i), s(a_i)$ 
the terminal vertex of $a_i$ and the source vertex of $a_i$,
respectively.
The graph $G_A$  
 has the $N_A \times N_A$
 transition matrix $A^G =[A^G(a_i,a_j)]_{i,j=1}^{N_A}$ of edges defined by 
\begin{equation}
A^G(a_i,a_j) =
\begin{cases} 
 1 &  \text{  if  } t(a_i) = s(a_j), \\
 0 & \text{  otherwise}
\end{cases}\label{eq:AG}
\end{equation}
for $a_i, a_j \in E_A$.
The Cuntz--Krieger algebra $\OA$ for the matrix $A$
is defined as the Cuntz--Krieger algebra
 ${\mathcal{O}}_{A^G}$ 
for the matrix $A^G$ which is the universal $C^*$-algebra generated by 
partial isometries
$S_{a_i}$ indexed by edges $a_i, i=1,\dots, N_A$ subject to the relations:
\begin{equation}
\sum_{j=1}^{N_A}  S_{a_j} S_{a_j}^* = 1, \qquad
S_{a_i}^* S_{a_i} = 
\sum_{j=1}^{N_A}  A^G(a_i, a_j ) S_{a_j} S_{a_j}^*
 \quad \text{ for } i=1,\dots,N_A.
\label{eq:OAG}
\end{equation}
The subalgebra 
$\DA$ is defined as the algebra 
$\mathcal{D}_{A^G}$.
The pair $(\OA,\DA)$ is called the Cuntz--Krieger pair
for the matrix $A$.
Since $1 \in \DA\subset\OA$,
the pair 
$(\OA,\DA)$ is relative $\sigma$-unital.
As in \cite{MaPacific},
the isomorphism class of the pair 
$(\OA,\DA)$ is exactly corresponding to the continuous orbit equivalence class
of the underlying one-sided topological Markov shift
$(X_A,\sigma_A)$. 
Its complete classification result has been obtained in \cite[Theorem 3.6]{MMKyoto}.

Let $A, B$ 
be irreducible square matrices with entries in nonnegative integers.
In \cite{Williams}, 
R. F. Williams  proved that 
the two-sided topological Markov shifts 
$(\bar{X}_A, \bar{\sigma}_A)$ and 
$(\bar{X}_B, \bar{\sigma}_B)$ 
are topologically conjugate
if and only if 
the matrices $A,B$ are strong shift equivalent.
Two nonnegative matrices $A, B$ are said to be  elementary equivalent
if there exist
nonnegative rectangular matrices $C,D$ such that 
$A = CD, B = DC$.
If there exists a finite sequence of nonnegative matrices $A_0,A_1,\dots, A_n$
such that
$A = A_0, B = A_n$ and $A_i$ is elementary equivalent to $A_{i+1}$ for 
$i=1,2,\dots, n-1$, 
then $A$ and $B$ are said to be strong shift equivalent
(\cite{Williams}).  
 Hence topological conjugacy of two-sided topological Markov shifts
is generated by a finite sequence of  elementary equivalence of underlying matrices. 
  Let us denote by $B_k(\bar{X}_A)$ 
the set of admissible words with length $k$
of the topological Markov shift $(\bar{X}_A, \bar{\sigma}_A)$.
Put $B_*(\bar{X}_A) =\cup_{k=0}^\infty B_k(\bar{X}_A).$

In this section we will first show the following proposition.
\begin{proposition}\label{prop:SSE}
Suppose that two nonnegative square matrices $A$ and $B$ 
are elementary equivalent
such that $A =CD$ and $B = DC$.
Then we have 
$(\OA,\DA)
\underset{\operatorname{RME}}{\sim}
(\OB,\DB)$.
\end{proposition}
\begin{proof}
Suppose that the size of $A$ is $N$  and 
that of $B$ is $M$ 
so that 
$C$ is an $N \times M$ matrix and 
$D$ is an $M \times N$ matrix, respectively. 
 We set the square matrix
$
Z =
\begin{bmatrix}
0 & C \\
D & 0
\end{bmatrix}
$
as block matrix, and we see
\begin{equation*}
Z^2 
=
\begin{bmatrix}
CD & 0 \\
0 & DC
\end{bmatrix}
=
\begin{bmatrix}
A & 0 \\
0 & B
\end{bmatrix}.
\end{equation*}
Let us consider the Cuntz--Krieger algebra 
$\OZ$
for the matrix $Z$.
Since $E_Z = E_C \cup E_D$,
we may write 
the canonical generating partial isometries of 
$\OZ$ 
as
$ S_c, S_d, c \in E_C, d \in E_D$ 
so that
$
\sum_{c \in E_C} S_c S_c^*
+
\sum_{d \in E_D} S_d S_d^*
=1$
and
\begin{equation*}
S_c^* S_c = \sum_{d \in E_D}Z(c,d) S_d S_d^*, \qquad
S_d^* S_d = \sum_{c \in E_C}Z(d,c) S_c S_c^*
\end{equation*}
for $c \in E_C, d \in E_D$.
Let us denote by
$S_a, a \in E_A$ 
(resp. 
$S_b, b \in E_B$)
the canonical generating partial isometries of 
$\OA$ (resp. $\OB$)
satisfying the relations \eqref{eq:CK}.
As 
$Z^2 = 
\begin{bmatrix}
A & 0 \\
0 & B
\end{bmatrix},
$
we have a bijective correspondence
$\varphi_{A,CD}$ from $E_A$ to a subset of $E_C \times E_D$
(resp.
$\varphi_{B,DC}$ from $E_B$ to a subset of $E_D \times E_C$
)
such that
$S_c S_d \ne 0$ 
(resp. $S_d S_c \ne 0$) 
if and only if 
$\varphi_{A,CD}(a) = cd$
(resp.
$\varphi_{B,DC}(b) = dc$)
 for some $a \in E_A$
(resp. $b \in E_B$),
we may identify $cd$ (resp. $dc$) 
with $a$ (resp. $b$)
through the map 
$\varphi_{A,CD}$ (resp. $\varphi_{B,DC}$).
We may then write 
$S_{cd} = S_a$
(resp. $S_{dc} = S_b$)
where $S_{cd}$ denotes $S_c S_d$ (resp. $S_{dc}$ denotes $S_d S_c$).
Define the projections in $\OZ$ by 
$P_A = \sum_{c \in E_C} S_c S_c^*$
and
$P_B = \sum_{d \in E_D} S_d S_d^*.$
Both of them belong to $\DZ$
and satisfy
$P_A + P_B =1$.
It has been shown in \cite{MaETDS2004} (cf. \cite{MaPre2016})
that 
\begin{equation}
P_A \OZ P_A = \OA, \qquad P_B \OZ P_B = \OB,
\qquad
 \DZ P_A = \DA, \qquad  \DZ P_B = \DB. \label{eq:4.1}
\end{equation}
We put
$X = P_A \OZ P_B$
which has a natural structure of 
$\OA-\OB$ imprimitivity bimodule 
under the identification
$P_A \OZ P_A = \OA, P_B \OZ P_B = \OB.$

We will prove that
$X$ becomes $(\OA, \DA)$--$(\OB, \DB)$-relative imprimitivity bimodule.
Put 
$E_C= \{ c_1, \dots,c_{N_C}\}$
and
$E_D= \{ d_1, \dots,d_{N_D}\}$
for the matrices $C$ and $D$ respectively.
For $k =1,\dots,N_D$, take $c(k) \in E_C$ such that 
$c(k) d_k \in B_2(X_Z)$ so that we have
\begin{equation*}
S_{c(k)}^*S_{c(k)} \ge S_{d_k} S_{d_k}^*. \label{eq:ckd}
\end{equation*}
Similarly 
for $l =1,\dots,N_C$, take $d(l) \in E_D$ such that 
$d(l) c_l \in B_2(X_Z)$ so that we have
\begin{equation*}
S_{d(l)}^*S_{d(l)} \ge S_{c_l} S_{c_l}^*. \label{eq:dlc}
\end{equation*}
We set
\begin{align*}
x_k & = S_{c(k)} S_{d_k} S_{d_k}^* \quad \text{ for } k=1,\dots, N_D, \\
y_l &=S_{d(l)} S_{c_l} S_{c_l}^* \qquad \text{for } l=1,\dots,N_D.
\end{align*}
For $d_1 \in \D_A, d_2 \in \D_B$,
we have 
\begin{align*}
{}_{\OA}\!\langle x_k d_2 \mid x_k \rangle
= & S_{c(k)} S_{d_k} S_{d_k}^* d_2 S_{d_k} S_{d_k}^*S_{c(k)}^* \in \DA, \\ 
\langle x_k \mid d_1 x_k \rangle_{\OB}
= &  S_{d_k} S_{d_k}^*S_{c(k)} d_1 S_{c(k)} S_{d_k} S_{d_k}^* \in \DB
\end{align*}
so that 
$x_k$ belongs to $X_D$ and similarly 
$y_l$ belongs to $X_D$.
We also have
\begin{align*}
\sum_{k=1}^{N_D} \langle x_k \mid x_k \rangle_{\OB}
= & \sum_{k=1}^{N_D} (S_{c(k)} S_{d_k} S_{d_k}^*)^*(S_{c(k)} S_{d_k} S_{d_k}^*) \\
= & \sum_{k=1}^{N_D} S_{d_k} S_{d_k}^* S_{c(k)}^*S_{c(k)} S_{d_k} S_{d_k}^* \\
= & \sum_{k=1}^{N_D} S_{d_k} S_{d_k}^* 
= P_B
\end{align*}
For $n\ne m$, we have
\begin{equation*}
{}_{\OA}\!\langle x_n d_2 \mid x_m \rangle
= S_{c(n)} S_{d_n} S_{d_n}^* d_2 S_{d_m} S_{d_m}^*S_{c(m)}^* =0.
\end{equation*}
Similarly we have
$
\sum_{l=1}^{N_C} {}_{\OA}\!\langle y_l \mid y_l \rangle = P_A
$
and
$\langle y_n  \mid d_1 y_m \rangle_{\OB} =0$ 
for $n\ne m$,
so that 
$X$ becomes $(\OA, \DA)$--$(\OB, \DB)$-relative imprimitivity bimodule.
\end{proof}

In \cite{PS}, Parry--Sullivan proved 
that the flow equivalence relation of topological Markov shifts 
is generated by  strong shift equivalences 
and expansions $A \rightarrow \tilde{A}$
defined bellow.

For an $N \times N$ matrix
$A = [A(i,j)]_{i,j=1}^N$
with entries in $\{0,1\}$,
put
\begin{equation}
\tilde{A} =
\begin{bmatrix}
0      & A(1,1) & \cdots & A(1,N) \\
1      & 0      & \cdots & 0      \\
0      & A(2,1) & \cdots & A(2,N) \\
\vdots & \vdots &        & \vdots \\
0      & A(N,1) & \cdots & A(N,N)
\end{bmatrix}, \label{eq:tildeA}
\end{equation}
which is called the expansion of $A$ at the vertex $1$.
The expansion of $A$ at other vertices are similarly defined.
\begin{proposition}\label{prop:expansion}
$(\OA,\DA)
\underset{\operatorname{RME}}{\sim}
(\OTA,\DTA)$.
\end{proposition}
\begin{proof}
Let 
$\{0,1,\dots,N\}$
be the set of symbols for the topological Markov shifts 
$(\bar{X}_{\tilde{A}},\bar{\sigma}_{\tilde{A}})$
defined by the matrix $\tilde{A}$.
Let us denote by 
$\tilde{S}_0, \tilde{S}_1,\dots, \tilde{S}_N$
the canonical generating partial isometries of the Cuntz--Krieger algebra
${\mathcal{O}}_{\tilde{A}}$
satisfying
$\sum_{j=0}^N
\tilde{S}_j\tilde{S}_j^* =1,
\tilde{S}_i^*\tilde{S}_i
= \sum_{j=0}^N
\tilde{A}(i,j) \tilde{S}_j\tilde{S}_j^*
$ 
for $i=0,1,\dots,N$.
Put
$
P = \sum_{i=1}^N
\tilde{S}_i\tilde{S}_i^*.
$
The  identities 
\begin{equation}
 \tilde{S}_1^* P \tilde{S}_1 
=\tilde{S}_1^* \tilde{S}_1 = \tilde{S}_0 \tilde{S}_0^*,\qquad
P + \tilde{S}_0 \tilde{S}_0^* =P + \tilde{S}_1^* P \tilde{S}_1 =1
\label{eq:PtildeS}
\end{equation}
hold, so that 
we have
\begin{equation}
P\OTA P = \OA, \qquad 
\DTA P = \DA. \label{eq:POAtilde}
\end{equation}
We put
$X = P\OTA$
which has a natural structure of 
$\OA-\OTA$ imprimitivity bimodule 
under the identification
$P\OTA P = \OA,  \DTA P = \DA.$
We will prove that
$X$ becomes $(\OA, \DA)$--$(\OTA, \DTA)$-relative imprimitivity bimodule.
We set
$x_1 = P,\,  x_2 = P  \tilde{S}_1$ and $y_1 = P$.
For $d_1 \in \DA, d_2 \in \DTA,$
we have
\begin{align*}
{}_{\OA}\!\langle x_1 d_2 \mid x_1 \rangle
= & P d_2 P \in \DA, \\
{}_{\OA}\!\langle x_2 d_2 \mid x_2 \rangle
= & P \tilde{S}_1 d_2 \tilde{S}_1^* P \in \DA, \\
\langle d_1 x_1 \mid x_1 \rangle_{\OTA}
= & P d_1 P \in \DA \subset \DTA, \\
\langle d_1 x_2 \mid x_2 \rangle_{\OTA}
= & \tilde{S}_1^* P  d_1 P \tilde{S}_1\in  \DTA, 
\end{align*}
so that 
$x_1, x_2, y_1 \in X_D$.
We also have
\begin{align*}
\sum_{k=1}^2 \langle x_k \mid x_k \rangle_{\OTA}
= & P^*P + (P \tilde{S}_1)^* (P \tilde{S}_1) = P + \tilde{S}_1^*P \tilde{S}_1 =1, \\
{}_{\OA}\!\langle x_1 d_2 \mid x_2 \rangle
= &P d_2 (P \tilde{S}_1)^* = Pd_2  \tilde{S}_1^* P=0, \\
{}_{\OA}\!\langle x_2 d_2 \mid x_1 \rangle
= & P \tilde{S}_1 d_2 P^*  = 0.
\end{align*}
Hence
$X$ becomes $(\OA, \DA)$--$(\OTA, \DTA)$-relative imprimitivity bimodule.
\end{proof}
We have thus obtained the following theorem.
\begin{theorem}\label{thm:FERME}
If two-sided topological Markov shifts
$(\bar{X}_A, \bar{\sigma}_A)$ and 
$(\bar{X}_B, \bar{\sigma}_B)$ 
are flow equivalent,
then the Cuntz--Krieger pairs
$(\OA,\DA)$ and $(\OB,\DB)$
are relatively Morita equivalent.
\end{theorem}

\section{Corner isomorphisms in Cuntz--Krieger pairs}
Let $A, B, Z$ be square irreducible non-permutation matrices 
with entries in nonnegative integers.
\begin{definition}
Two  Cuntz--Krieger pairs 
$(\OA,\DA)$ and $(\OZ,\DZ)$
are said to be {\it elementary corner isomorphic}
if there exists a projection
$P \in \DZ$
such that 
\begin{equation}
P \OZ P = \OA, \qquad \DZ P = \DA. \label{eq:corner}
\end{equation}
Two  Cuntz--Krieger pairs 
$(\OA,\DA)$ and $(\OB,\DB)$
are said to be {\it corner isomorphic}
if there exists a finite chain of 
Cuntz--Krieger pairs 
$({\mathcal{O}}_{Z_i}, {\mathcal{D}}_{Z_i}), i=0,1,\dots,n$ 
such that
$Z_0 = A, Z_n = B$,
and 
either 
 $({\mathcal{O}}_{Z_i}, {\mathcal{D}}_{Z_i})$
and
$({\mathcal{O}}_{Z_{i+1}}, {\mathcal{D}}_{Z_{i+1}})$ 
or
$({\mathcal{O}}_{Z_{i+1}}, {\mathcal{D}}_{Z_{i+1}})$ 
and
 $({\mathcal{O}}_{Z_i}, {\mathcal{D}}_{Z_i})$
are elementary corner isomorphic
for all
$i=0,1,\dots,n$.
That is,
the equivalence relation  generated by 
elementary corner isomorphisms in Cuntz--Krieger pairs
is  the corner isomorphism. 
\end{definition}
We will prove the following theorem.
\begin{theorem}\label{thm:corner}
If two  Cuntz--Krieger pairs 
$(\OA,\DA)$ and $(\OB,\DB)$
are corner isomorphic,
then
there exists an isomorphism 
$\Phi: \OA\otimes\K \longrightarrow \OB\otimes \K$ of $C^*$-algebras
such that 
$\Phi(\DA\otimes\C)=\DB\otimes \C$.
\end{theorem}
\begin{proof}
We use the notation $Z$ of matrix instead of $B$,
so that we suppose  that 
$(\OA,\DA)$ and $(\OZ,\DZ)$ are elementary corner isomorphic
by a projection $P \in \DZ$ satisfying \eqref{eq:corner}.
Although by showing that
$X = P\OZ$ is an $(\OA,\DA)$--$(\OZ,\DZ)$-relative imprimitivity bimodule,
we know that 
$(\OA,\DA)$ and $(\OZ,\DZ)$ are relatively Morita equivalent,
and hence 
there exists an isomorphism 
$\Phi: \OA\otimes\K \longrightarrow \OB\otimes \K$ of $C^*$-algebras
such that 
$\Phi(\DA\otimes\C)=\DB\otimes \C$,
we will directly construct such a isomorphism $\Phi$ in the following way.

We may assume that the projection $Q = 1-P $ 
is not zero.
Let $S_1, \dots, S_{N_Z}$ be the canonical generating partial isometries
of the Cuntz--Krieger algebra $\OZ$.
As $Q \in \DZ$, 
one may find a finite family of admissible words
$\mu(k) \in B_*(X_Z), k=1,\dots,N_1$ 
such that 
$|\mu(1)|= \cdots =|\mu(N_1)|$
and
$Q = \sum_{k=1}^{N_1}S_{\mu(k)}S_{\mu(k)}^*$.
Since $Z$ is irreducible, we may find 
admissible words
$\nu(k)  \in  B_*(X_Z)$
for each $\mu(k)$ such that
$|\nu(1)|= \cdots =|\nu(N_1)|$
and
\begin{equation*}
P \ge S_{\nu(k)}S_{\nu(k)}^*, \qquad
S_{\nu(k)}S_{\mu(k)} \ne 0, \qquad
k=1,\dots,N_1.
\end{equation*}
As $\nu(k) \mu(k)$ is an admissible word in $X_Z$, we know 
$
S_{\nu(k)}^*S_{\nu(k)}\ge S_{\mu(k)}S_{\mu(k)}^*.
$
Put
$$
U_k = S_{\nu(k)} S_{\mu(k)}S_{\mu(k)}^*, \qquad
k=1,\dots,N_1.
$$
Then we have
\begin{gather*}
\sum_{k=1}^{N_1} U_k^* U_k
 = \sum_{k=1}^{N_1} S_{\mu(k)}S_{\mu(k)}^*S_{\nu(k)}^*
S_{\nu(k)} S_{\mu(k)}S_{\mu(k)}^*
 = \sum_{k=1}^{N_1}  S_{\mu(k)}S_{\mu(k)}^*
= Q, \\
\intertext{and}
U_k U_k^*\le   S_{\nu(k)}S_{\nu(k)}^* \le P.
\end{gather*}
The sequence further satisfies the following
\begin{equation*}
U_k U_l^* =0\quad\text{ for } k\ne l,
\qquad
U_k \DZ U_k^* \subset \DZ P = \DA,
\qquad
U_k^* \DZ U_k \subset \DZ Q \subset \DZ.
\end{equation*}
As in the proof of Theorem \ref{RMEPHI},
by setting 
\begin{equation*}
u_n = \sum_{k=1}^{N_1} U_k \otimes s_{n_k,n}, 
 \qquad w_n = P \otimes s_{{n_0},n} +  u_n,
\qquad n=1,2,\dots.
\end{equation*}
we have a sequence $w_n, n\in \N$ 
of partial isometries in $M(\OZ\otimes \K)$ such that 
\begin{enumerate}
\item $w_n^* w_n = 1 \otimes f_n$.
\item $w_n w_n^* \le P\otimes f_n$. 
\item $w_n (\DZ \otimes\C)w_n^* \subset \DZ P\otimes\C$.
\item $w_n^* (\DZ\otimes\C) w_n \subset \DZ Q\otimes\C$.
\end{enumerate}
Let $f_{n,m}$ be a partial isometry satisfying
$f_{n,m}^*f_{n,m} = f_m,\, f_{n,m}f_{n,m}^* = f_n.$
We put
\begin{align*}
v_1 &  = w_1 = P \otimes s_{1_0,1} + u_1, \\
v_{2n} & = (P \otimes f_n - v_{2n-1}v_{2n-1}^*)(p \otimes f_{n,n+1})
\quad \text{ for } 1\le n \in \N, \\
v_{2n-1} & = w_n(1\otimes f_n - v_{2n-2}^*v_{2n-2}) \quad \text{ for } 2\le n \in \N. 
\end{align*} 
By the same way as Lemma \ref{lem:3.4},
we have
for $n\in \N$ 
\begin{enumerate}
\item$v_{2n-2}^* v_{2n-2} + v_{2n-1}^* v_{2n-1} = 1\otimes f_n$.
\item$v_{2n-1} v_{2n-1}^* + v_{2n} v_{2n}^* = P\otimes f_n$.
\item $v_n(\DZ \otimes \C)v_n^* \subset \DZ P\otimes\C.$
\item $v_n^*(\DZ P \otimes \C)v_n \subset \DZ\otimes\C.$
\end{enumerate}
Hence  
the summation
$\sum_{n=1}^\infty v_n$
 converges in $M(\OZ\otimes\K)$ 
to certain partial isometry written $V_A$
 in the strict topology of  $M(\OZ\otimes \K)$,
so that 
{\begin{enumerate}
\item$V_A^* V_A = 1\otimes 1$.
\item$V_A V_A^* = P\otimes 1.$
\item$V_A(\DZ\otimes\C)V_A^* = \DZ P\otimes\C.$
\item$V_A^*(\DZ P\otimes\C)V_A = \DZ\otimes\C.$
\end{enumerate}}
Putting
$\Phi_A = \Ad(V_A): \OZ\otimes\K\longrightarrow \SOA$
which is an isomorphism between $\SOZ$ and $\SOA$
such that 
$\Phi(\DZ\otimes\C)=\DA\otimes \C$.
\end{proof}
Therefore we have  the following theorem.
\begin{theorem} \label{thm:iffcorner}
The Cuntz-Krieger pairs
$(\OA,\DA)$ and $(\OB,\DB)$ are corner isomorphic
if and only if
there exists an isomorphism
$\Phi: \OA\otimes\K \longrightarrow \OB\otimes \K$ of $C^*$-algebras
such that 
$\Phi(\DA\otimes\C)=\DB\otimes \C$.
\end{theorem}
\begin{proof}
The only if part follows from Theorem \ref{thm:corner}.
We will show the if part.
Suppose that
there exists an isomorphism
$\Phi: \OA\otimes\K \longrightarrow \OB\otimes \K$ of $C^*$-algebras
such that 
$\Phi(\DA\otimes\C)=\DB\otimes \C$.
By \cite{MMKyoto},
the two-sided topological Markov shifts 
$(\bar{X}_A, \bar{\sigma}_A)$ and 
$(\bar{X}_B, \bar{\sigma}_B)$ 
are flow equivalent, so that
the two matrices are connected by a finite chain of
strong shift equivalences and symbol expansions.
In the proofs of Proposition \ref{prop:SSE}
and Proposition \ref{prop:expansion},
we know that
$(\OA,\DA)$ and $(\OB,\DB)$ are corner isomorphic.
\end{proof}
Therefore we may summarize our discussions in the following way.
\begin{theorem}\label{thm:sumCK}
Let $A, B$ be irreducible non-permutation matrices with entries in $\{0,1\}$.  
Let
$\OA, \OB$ be the associated Cuntz--Krieger algebras.
Then the following assertions are mutually equivalent:
\begin{enumerate}
\item 
$(\OA, \DA)
\underset{\operatorname{RME}}{\sim}
(\OB, \DB).
$
\item 
$(\OA\otimes\K, \DA\otimes\C)
\underset{\operatorname{RME}}{\sim}
(\OB\otimes\K, \DB\otimes\C).
$
\item 
There exists an isomorphism 
$\Phi: \OA\otimes\K \longrightarrow \OB\otimes \K$ of $C^*$-algebras
such that 
$\Phi(\DA\otimes\C)=\DB\otimes \C$.
\item
$(\OA, \DA)$
and
$
(\OB, \DB)
$
are corner isomorphic.
\item
The two-sided topological Markov shifts 
$(\bar{X}_A, \bar{\sigma}_A)$ and 
$(\bar{X}_B, \bar{\sigma}_B)$ 
are flow equivalent.
\end{enumerate}
\end{theorem} 
\begin{proof}
(1) $\Longleftrightarrow$ (2) comes from Lemma \ref{lem:stab}.

(1) $\Longleftrightarrow$ (3) comes from Theorem \ref{thm:RMEPHI2}.

(3) $\Longleftrightarrow$ (4) comes from Theorem \ref{thm:iffcorner}.

(5) $\Longrightarrow$ (1) comes from Theorem \ref{thm:FERME}. 

(3) $\Longrightarrow$ (5) comes from \cite[Corollary 3.8]{MMKyoto}.
\end{proof}
We note that the implication (5) $\Longrightarrow$ (3) is seen in \cite{CK},
and that the equivalence between (3) and (5)  is seen in \cite[Corollary 3.8]{MMKyoto}.

\section{Relative Picard groups}
Let $(\A_1, \D_1)$ and $(\A_2, \D_2)$ be
relative $\sigma$-unital pairs of $C^*$-algebras.
Let
$X, Y$ be
$(\A_1, \D_1)$--$(\A_2, \D_2)$-relative imprimitivity bimodule.
Then $X$ and $Y$ are said to be {\it equivalent}\/
if there exists an isomorphism
$\varphi: X\longrightarrow Y$ of
$\A_1$--$\A_2$-imprimitivity bimodule
such that 
$$
\langle \varphi(x_1) \mid \varphi(x_2) \rangle 
= \langle x _1\mid x_2 \rangle
\quad \text{ for }
\quad
x_1, x_2 \in X
$$
for both left and right inner products.
As $\varphi: X\longrightarrow Y$ 
preserves the bimodule structures and inner products
of $X$ and $Y$,
we know 
$\varphi(X_D)=Y_D.$
We denote by $[X]$ the equivalence class of 
relative imprimitivity bimodule.
For a relative $\sigma$-unital pair $(\A,\D)$ of $C^*$-algebras,
we define a relative version of Picard group as follows.
\begin{definition}
The {\it relative Picard group}\/ $\Pic(\A,\D)$ for $(\A,\D)$ 
is defined by the group
of equivalence classes $[X]$ of 
$(\A, \D)$--$(\A, \D)$-relative imprimitivity bimodule
by the product
\begin{equation*}
[X]\cdot [Y] := [X\otimes_\A Y].
\end{equation*}
\end{definition}
We remark that the identity element of the group
$\Pic(\A,\D)$  is the class of the identity  
$(\A, \D)$--$(\A, \D)$-relative imprimitivity bimodule
$X =\A$ defined by the module structure and the inner products:
\begin{equation}
a\cdot x \cdot b = axb, \qquad
{}_\A\!\langle x \mid y \rangle:= x y^*,\qquad
\langle x \mid y \rangle_{\A}: = x^* y
\quad \text{ for }\quad a, b, x,  y \in \A. 
\end{equation}
Since $(\A,\D)$ is relative $\sigma$-unital,
the above $X$ becomes an $(\A, \D)$--$(\A, \D)$-relative imprimitivity bimodule
as seen in Lemma \ref{lem:theta}.
\begin{lemma}\label{lem:RNEPIC}
If
$
(\A_1, \D_1)\underset{\operatorname{RME}}{\sim}(\A_2, \D_2),
$
we have
$
\Pic(\A_1, \D_1)=\Pic(\A_2, \D_2).
$
Hence
we have
$
\Pic(\A, \D)=\Pic(\A\otimes\K, \D\otimes\C)
$
for every 
relative $\sigma$-unital pair $(\A,\D)$ of $C^*$-algebras.
\end{lemma}
\begin{proof}
Let
$X$ be
$(\A_1, \D_1)$--$(\A_2, \D_2)$-relative imprimitivity bimodule,
and $\bar{X}$ its conjugate module,
which is 
$(\A_2, \D_2)$--$(\A_1, \D_1)$-relative imprimitivity bimodule.
It is easy to see that the correspondence
$$
[Y] \in \Pic(\A_1, \D_1) \longrightarrow
[\bar{X}\otimes_{\A_1}Y\otimes_{\A_1}X]\in \Pic(\A_2, \D_2)
$$
yields an isomorphism as groups,
because
$[\bar{X}\otimes_{\A_1}X] $ is the unit of the group $\Pic(\A_2, \D_2)$
and
$[X\otimes_{\A_2}\bar{X}]$ is the unit of the group $\Pic(\A_1, \D_1)$.
\end{proof}
If $\theta:\A_1\longrightarrow \A_2$
is an isomorphism of $C^*$-algebras such that 
$\theta(\D_1) = \D_2$,
then we write
$\theta:(\A_1, \D_1) \longrightarrow (\A_2, \D_2)$
and call an isomorphism of relative $\sigma$-unital pairs of $C^*$-algebras.
As in Lemma \ref{lem:theta},
any isomorphism 
$\theta:(\A_1, \D_1) \longrightarrow (\A_2, \D_2)$
gives rise to a $(\A_1, \D_1)$--$(\A_2, \D_2)$-relative imprimitivity bimodule
$X_\theta$.
The following lemma is straightforward.
\begin{lemma}
Let
$\theta_{12}:(\A_1, \D_1) \longrightarrow (\A_2, \D_2)$
and
$\theta_{23}:(\A_2, \D_2) \longrightarrow (\A_3, \D_3)$
be isomorphisms of relative $\sigma$-unital pairs of $C^*$-algebras.
Then we have
$$
[X_{\theta_{12}} \otimes_{\A_2}X_{\theta_{23}}] = [X_{\theta_{23} \circ \theta_{12}}]. 
$$
Therefore we have a contravariant functor from the category of relative $\sigma$-unital
$C^*$-algebras with isomorphisms
$\theta:(\A_1, \D_1) \longrightarrow (\A_2, \D_2)$
as morphisms into the category of relative $\sigma$-unital
$C^*$-algebras with equivalence classes of relative imprimitivity bimodules.
\end{lemma}
Let $\Aut(\A,\D)$ 
be the group of automorphisms
$\theta$ on $\A$ such that 
$\theta(\D) = \D$, that is,
$$
\Aut(\A,\D) := \{ \theta \in \Aut(\A) \mid \theta(\D) = \D\}
$$
We denote by
$U(\A,\D)$
the group of unitaries 
 $u \in M(\A)$ satisfying 
$u\D u^* = \D$.
We denote by $\Ad(u)$ the automorphism of $(\A,\D)$
defined by
$\Ad(u)(a) = u a u^*$ for $a \in \A$.
Let us denote by
$\Int(\A,\D)$
the subgroup of 
$\Aut(\A,\D)$ consisting of such automorphisms of $(\A,\D)$.
Hence $\Int(\A,\D)$ is a normal subgroup of 
$\Aut(\A,\D)$.
By the preceding lemma, we have an anti-homomorphism
\begin{equation*}
\theta\in\Aut(\A,\D) \longrightarrow [X_\theta]\in \Pic(\A,\D).
\end{equation*}
The following proposition and its corollary
are achieved by a similar manner to
Brown--Green--Rieffel's argument \cite[Proposition 3.1]{BGR}
and {\cite[Corollary 3.2]{BGR}}.
\begin{proposition}[{cf. \cite[Proposition 3.1]{BGR}}]
The kernel of the anti-homomorphism
from $\Aut(\A,\D)$ into
$\Pic(\A,\D)$ is exactly $\Int(\A,\D)$.
That is, we have an exact sequence:
\begin{equation*}
1 
\longrightarrow \Int(\A,\D)
\longrightarrow \Aut(\A,\D)
\longrightarrow \Pic(\A,\D).
\end{equation*} 
\end{proposition}
\begin{corollary}[{cf. \cite[Corollary 3.2]{BGR}}] \label{cor:3.2}
Let $(\A_1, \D_1)$ and $(\A_2, \D_2)$ be
relative $\sigma$-unital pairs of $C^*$-algebras.
Let $\alpha, \beta:(\A_1, \D_1) \longrightarrow(\A_2, \D_2)$
be isomorphisms.
If $X_\alpha$ and $X_\beta$ are equivalent,
then there exists a unitary $u \in U(\A,\D)$
such that 
$\beta = \Ad(u)\circ \alpha$.
\end{corollary}
The following lemma is also a relative version
of {\cite[Lemma 3.3]{BGR}}.
\begin{lemma}[{cf. \cite[Lemma 3.3]{BGR}}]\label{lem:3.3}
Let $X$ be an $(\A_1, \D_1)$--$(\A_2, \D_2)$-relative imprimitivity bimodule.
Let $(\A_0,\D_0)$ be the linking pair 
of $X$ defined by \eqref{eq:linkinga} and \eqref{eq:linkingd}.
Then $X$ is equivalent to $X_\theta$ for some 
isomorphism
$\theta:(\A_1,\D_1)\longrightarrow (\A_2,\D_2)$
if and only if
there exists a partial isometry $v \in M(\A_0)$ such that 
\begin{gather}
v^* v =
\begin{bmatrix}
1 & 0\\
0 & 0
\end{bmatrix}, 
\qquad
 v v^* =
\begin{bmatrix}
0 & 0\\
0 & 1
\end{bmatrix}  \label{eq:vvvv}\\
\intertext{and}
v\D_0 v^* = \D_0 vv^*,\qquad
v^*\D_0 v = \D_0 v^*v. \label{eq:vdv}
\end{gather}
In this case, $\theta$ is defined by $\theta(a) = v a v^*, \, a \in \A_1$. 
\end{lemma}
\begin{remark}
Under the equality \eqref{eq:vvvv},
the second equality of \eqref{eq:vdv} follows 
from the first equality of \eqref{eq:vdv}.
Because the first one of \eqref{eq:vdv}
ensures us the equality
\begin{equation}
v^*v\D_0 v^*v = v^*\D_0 vv^*v. \label{eq:v*vd}
\end{equation}
By \eqref{eq:vvvv},
$v^*v $ commutes with any elements of $\D_0$
so that \eqref{eq:v*vd} goes to
the second equality of \eqref{eq:vdv}.
\end{remark}
{\it Proof of Lemma \ref{lem:3.3}.}
Although the proof basically follows the proof of \cite[Lemma 3.3]{BGR},
we give it for the sake of completeness.
Suppose that 
 $X$ is equivalent to $X_\theta$ for some 
isomorphism
$\theta:(\A_1,\D_1)\longrightarrow (\A_2,\D_2)$.
By this isomorphism, the linking algebra $\A_0$ of $X$
is identified with that of $X_\theta$.
Hence $X_\theta = \A_1$ and
\begin{equation*}
\A_0 
 = \{
\begin{bmatrix}
a_1      &  x  \\
\bar{y} & a_2
\end{bmatrix}
\mid 
a_1 \in \A_1, a_2 \in \A_2, x \in X_\theta, \bar{y} \in \bar{X}_\theta \}.
\end{equation*}
We define operators
$v, v^* $ on $X_\theta\oplus\A_2$ by 
\begin{equation*}
v
\begin{bmatrix}
z\\
c_2
\end{bmatrix}
=
\begin{bmatrix}
0\\
\theta(z)
\end{bmatrix}, \qquad
v^*
\begin{bmatrix}
z\\
c_2
\end{bmatrix}
=
\begin{bmatrix}
\theta^{-1}(c_2)\\
0
\end{bmatrix}
\quad
\text{ for }
z \in X_\theta, c_2 \in \A_2
\end{equation*}
where
$X_\theta =\A_1$ so that $\theta(z) \in \A_2$.
Put
\begin{equation*}
R_v(
\begin{bmatrix}
a_1      &  x  \\
\bar{y} & a_2
\end{bmatrix}
)
=\begin{bmatrix}
a_1      &  x  \\
\bar{y} & a_2
\end{bmatrix}
v, \qquad
L_v(
\begin{bmatrix}
a_1      &  x  \\
\bar{y} & a_2
\end{bmatrix}
)
=
v
\begin{bmatrix}
a_1      &  x  \\
\bar{y} & a_2
\end{bmatrix}.
\end{equation*}
It is straightforward to see that 
\begin{equation*}
R_v(
\begin{bmatrix}
a_1      &  x  \\
\bar{y} & a_2
\end{bmatrix}
)
\begin{bmatrix}
a'_1      &  x' \\
\bar{y}' & a'_2
\end{bmatrix}
=\begin{bmatrix}
a_1      &  x  \\
\bar{y} & a_2
\end{bmatrix}
L_v(
\begin{bmatrix}
a'_1      &  x'  \\
\bar{y}' & a'_2
\end{bmatrix}
).
\end{equation*}
Hence the pair 
$(L_v, R_v)$ defines an element of $M(\A_0)$
as a double centralizer of $\A_0$.
Similarly 
$(L_{v^*}, R_{v^*})$ defines an element of $M(\A_0)$
such that 
$(L_v, R_v)^* =(L_{v^*}, R_{v^*}),$
so that we may write
$(L_v, R_v)=v$.
It then follows that 
\begin{align*}
v^* v 
\begin{bmatrix}
z\\
c_2
\end{bmatrix}
&=
\begin{bmatrix}
z\\
0
\end{bmatrix}
\quad
\text{ and hence }
\quad
v^* v 
=
\begin{bmatrix}
1 & 0\\
0 & 0
\end{bmatrix}, \\
v v^* 
\begin{bmatrix}
z\\
c_2
\end{bmatrix}
&=
\begin{bmatrix}
0\\
c_2
\end{bmatrix}
\quad
\text{ and hence }
\quad
v v^* 
=
\begin{bmatrix}
0 & 0\\
0 & 1
\end{bmatrix}.
\end{align*}
It is direct to see that 
\begin{equation*}
v 
\begin{bmatrix}
a_1 & 0\\
0 & 0
\end{bmatrix}
v^*
= 
\begin{bmatrix}
0 & 0\\
0 & \theta(a_1)
\end{bmatrix}.
\end{equation*}
This means that $\theta(a_1) = v a_1 v^*$ for $a_1 \in \A_1$
under the identification between 
$a_1 $ and 
$
\begin{bmatrix}
a_1 & 0\\
0 & 0
\end{bmatrix}
$
for $a_1 \in \A_1$.
Since $\theta:\A_1\longrightarrow \A_2$
satisfies
$\theta(\D_1) = \D_2$
and
$\D_1 = \D_0 v^*v,$
$\D_2 = \D_0 vv^*$,
we have
$$
v\D_0 v^* =v\D_1v^* =\theta(\D_1) =\D_2 
=\D_0 vv^*
$$
and
$v^* \D_0 v = \D_0 v^*v$.

Conversely
suppose that a partial isometry $v \in M(\A_0)$ satisfies
the equalities \eqref{eq:vvvv} and \eqref{eq:vdv}.
It is easy to see that there exists an element
$\theta(a)$ in $\A_2$ for each $a \in \A_1$ such that  
 \begin{equation*}
v 
\begin{bmatrix}
a & 0\\
0 & 0
\end{bmatrix}
v^*
= 
\begin{bmatrix}
0 & 0\\
0 & \theta(a)
\end{bmatrix}
\end{equation*}
and the correspondence
$a \in\A_1 \longrightarrow \theta(a)\in \A_2$
gives rise to an isomorphism of $C^*$-algebras.
The conditions \eqref{eq:vvvv} and \eqref{eq:vdv} implies that
$v \D_0 v^* = \D_0 vv^* = \D_2$
and
$v^* \D_0 v = \D_0 v^*v = \D_1$
so that we have
$v\D_1 v^* = vv^* \D_0 vv^* = \D_2$.
This implies that 
$\theta(\D_1) = \D_2$. 
 
We will next show that 
$X$ is equivalent to $X_\theta$.
We identify $\A_1$ with its image in $\A_0$
and then we will define a map
$\eta:X\longrightarrow \A_1 (=X_\theta)$
 by
\begin{equation*}
\eta(x) := 
\begin{bmatrix}
0 & x\\
0 & 0
\end{bmatrix}
v
\qquad \text{ for }
x \in X.
\end{equation*}
Since 
\begin{equation*}
v^* v \eta(x) v^*v
= 
\begin{bmatrix}
1 & 0\\
0 & 0
\end{bmatrix}
\begin{bmatrix}
0 & x\\
0 & 0
\end{bmatrix}
vv^*v
=
\begin{bmatrix}
0 & x\\
0 & 0
\end{bmatrix}
v
=\eta(x),
\end{equation*}
we see that 
$\eta(x) \in \A_1$.
By a routine calculation,
we know that 
$\eta$ is a bimodule homomorphism
from $X$ to $X_\theta$ which
preserves both inner products,
and hence $\eta$ gives rise to an isomorphism
between $X$ and $X_\theta$. 
\qed

The following theorem is also a relative version of a Brown-Green-Rieffel' s theorem
We will give its proof for the sake of completeness.
\begin{theorem}[{cf. \cite[Theorem 3.4]{BGR}}]
Let
$(\A_1, \D_1)$ and $(\A_2, \D_2)$
be relative $\sigma$-unital pairs of $C^*$-algebras.
Let
$X$ be an
$(\A_1\otimes\K, \D_1\otimes\C)$--$(\A_2\otimes\K, \D_2\otimes\C)$-relative imprimitivity bimodule.
Then there exists
an isomorphism $\theta: \A_1\otimes\K\longrightarrow \A_2\otimes\K$
satisfying $\theta(\D_1\otimes\C) =\D_2\otimes\C$
such that $X$ is equivalent to $X_\theta$.
Furthermore $\theta$ is unique up to left multiplication
by an element of $\Int(\A_2\otimes\K,\D_2\otimes\C)$, that is 
if $X$ is equivalent to $X_\varphi$ for some isomorphism
$\varphi:
(\A_1\otimes\K, \D_1\otimes\C) 
\longrightarrow
(\A_2\otimes\K, \D_2\otimes\C),$
then
there exists a unitary
$u \in U(\A_2\otimes\K,\D_2\otimes\C)$
such that 
$\varphi = \Ad(u)\circ \theta$.
\end{theorem}
\begin{proof}
The uniqueness follows immediately from 
Corollary \ref{cor:3.2}.

Now let
$X$ be an
$(\A_1\otimes\K, \D_1\otimes\C)$--$(\A_2\otimes\K, \D_2\otimes\C)$-relative imprimitivity bimodule.
We put
$\bar{\A}_i = \A_i\otimes\K, \, 
\bar{\D}_i = \D_i\otimes\K
$
for $i=1,2$.
Let $(\bar{\A}_0,\bar{\D}_0)$ be the linking pair 
for $X$  defined from $\bar{\A}_i, \bar{\D}_i,i=1,2$ and $X$  
by \eqref{eq:linkinga} and \eqref{eq:linkingd}.
By the assumption 
that
$
(\bar{\A}_1, \bar{\D}_1)\underset{\operatorname{RME}}{\sim}(\bar{\A}_2, \bar{\D}_2)
$
with Theorem \ref{thm:RMEPHI2},
Proposition \ref{prop:V}
tells us that there exist
$v_i \in M(\bar{\A}\otimes\K), i=1,2$
such that 
\begin{align*}
v_i^* v_i  & = 1 \otimes 1 \quad \text{ in} \quad M(\bar{\A}_0\otimes\K), \quad i=1,2, \\
v_1 v_1^*  & = P_1 \otimes 1 \quad 
\text{where } 
P_1 =
\begin{bmatrix}
1 & 0\\
0 & 0
\end{bmatrix}
\quad
\text{ in }  M(\bar{\A}_0) \\
v_2 v_2^*  & = P_2 \otimes 1 \quad 
\text{where } 
P_2 =
\begin{bmatrix}
0 & 0\\
0 & 1
\end{bmatrix}
\quad
\text{ in }  M(\bar{\A}_0) \\
\intertext{ and }
v_i(\bar{D}_0 \otimes\C)v_i^* & = \bar{D}_i \otimes\C,\qquad
v_i^*(\bar{D}_i \otimes\C)v_i  = \bar{D}_0 \otimes\C, \qquad i =1,2.
\end{align*}
Put a partial isometry
$w = v_2 v_1^* \in M(\bar{\A}\otimes\K)$
so that 
we have 
\begin{gather*}
w^* w  
= P_1 \otimes 1
=
\begin{bmatrix}
1\otimes 1 & 0\\
0 & 0
\end{bmatrix},
\qquad
w w^*   
= P_2 \otimes 1 
 =
\begin{bmatrix}
0 & 0\\
0 & 1\otimes 1
\end{bmatrix}
\quad
\text{ in }  M(\bar{\A}_0\otimes\K) .
\\
\intertext{ and }
w(\bar{D}_1 \otimes\C) w^*  = \bar{D}_2 \otimes\C,\qquad
w^*(\bar{D}_2 \otimes\C) w  = \bar{D}_1 \otimes\C.
\end{gather*}
Let $p\in \C$ be the rank one projection $p =e_{1,1}$,
so that $\bar{\A}_1\otimes p\otimes\K$ is a corner of 
  $\bar{\A}_1\otimes \K\otimes\K$.
Hence by \cite[Lemma 2.5]{Brown}
there exists a partial isometry
$\bar{v}_1 \in M(\bar{A}_1\otimes\K\otimes\K)$
such that 
$$
\bar{v}_1^*\bar{v}_1 = 1 \otimes 1\otimes 1, \qquad
\bar{v}_1\bar{v}_1^* = 1 \otimes p\otimes 1.
$$
By the construction of 
$\bar{v}_1$,
we see that
$$
\bar{v}_1 ( \bar{D}_1\otimes \C\otimes\C) \bar{v}_1^*
 = \bar{D}_1 \otimes p\otimes \C, \qquad
\bar{v}_1^* ( \bar{D}_1\otimes p\otimes\C) \bar{v}_1
 = \bar{D}_1 \otimes \C\otimes \C.
$$
We can identify 
$\bar{A}_1$ and $\bar{D}_1$  
with
$\bar{A}_1\otimes 1 \otimes\K$
and
$\bar{D}_1\otimes 1 \otimes\C$,
respectively, so that 
we have
$\bar{v}_1\in M(\bar{A}_1\otimes\K)$ and
\begin{gather*}
\bar{v}_1^*\bar{v}_1 = 1 \otimes 1, \qquad
\bar{v}_1\bar{v}_1^* = 1 \otimes p, \\
\bar{v}_1 ( \bar{D}_1\otimes \C) \bar{v}_1^*
 = \bar{D}_1 \otimes p, \qquad
\bar{v}_1^* ( \bar{D}_1\otimes p) \bar{v}_1
 = \bar{D}_1 \otimes \C.
\end{gather*}
 Similarly we have
$\bar{v}_2\in M(\bar{A}_2\otimes\K)$ and
\begin{gather*}
\bar{v}_2^*\bar{v}_2 = 1 \otimes 1, \qquad
\bar{v}_2\bar{v}_2^* = 1 \otimes p, \\
\bar{v}_2 ( \bar{D}_2\otimes \C) \bar{v}_2^*
 = \bar{D}_2 \otimes p, \qquad
\bar{v}_2^* ( \bar{D}_2\otimes p) \bar{v}_2
 = \bar{D}_2 \otimes \C.
\end{gather*}
Define 
$\bar{v}\in M(\bar{A}_0\otimes\K)$ by
$$
\bar{v} =
\begin{bmatrix}
0 & 0\\
0 & \bar{v}_2
\end{bmatrix}
w
\begin{bmatrix}
\bar{v}_1^* & 0\\
0 & 0
\end{bmatrix}
\quad
\text{ in }
M(\bar{A}_0\otimes\K).
$$
We then have 
\begin{align*}
\bar{v}^*\bar{v}
& =
\begin{bmatrix}
\bar{v}_1 & 0\\
0 & 0
\end{bmatrix}
w^*
\begin{bmatrix}
0 & 0\\
0 & 1\otimes 1
\end{bmatrix}
w
\begin{bmatrix}
\bar{v}_1^* & 0\\
0 & 0
\end{bmatrix} \\
& =
\begin{bmatrix}
\bar{v}_1 & 0\\
0 & 0
\end{bmatrix}
w^*
w
\begin{bmatrix}
\bar{v}_1^* & 0\\
0 & 0
\end{bmatrix} \\
& = 
\begin{bmatrix}
1\otimes p & 0 \\
0 & 0 
\end{bmatrix}
\intertext{and}
\bar{v}\bar{v}^*
& =
\begin{bmatrix}
0 & 0\\
0 & \bar{v}_2
\end{bmatrix}
w
\begin{bmatrix}
1\otimes 1 & 0\\
0 & 0
\end{bmatrix}
w^*
\begin{bmatrix}
0 & 0\\
0 & \bar{v}_2^*
\end{bmatrix} \\
& =
\begin{bmatrix}
0 & 0\\
0 & \bar{v}_2
\end{bmatrix}
w
w^*
\begin{bmatrix}
0 & 0\\
0 & \bar{v}_2^*
\end{bmatrix} \\
& =
\begin{bmatrix}
0 & 0 \\
0 & 1\otimes p 
\end{bmatrix}.
\end{align*}
We will next show that 
$
\bar{v} ( \bar{D}_0\otimes p) \bar{v}^*
 = (\bar{D}_0 \otimes p) \bar{v}\bar{v}^*.
$
For
$
\begin{bmatrix}
d_1 & 0 \\
0 & d_2 
\end{bmatrix}
\in \bar{\D}_0
$
with
$d_i \in \bar{\D}_i, i=1,2,$
we have
\begin{equation*}
\bar{v}
\begin{bmatrix}
d_1\otimes p & 0 \\
0 & d_2\otimes p 
\end{bmatrix}
\bar{v}^*
 =
\begin{bmatrix}
0 & 0\\
0 & \bar{v}_2
\end{bmatrix}
w
\begin{bmatrix}
\bar{v}_1^*(d_1\otimes p)\bar{v}_1 & 0\\
0 & 0
\end{bmatrix}
w^*
\begin{bmatrix}
0 & 0\\
0 & \bar{v}_2^*
\end{bmatrix}.
\end{equation*}
Since
$\bar{v}_1^*(d_1\otimes p)\bar{v}_1\in \bar{D}_1\otimes\C$,
we have
$
w
\begin{bmatrix}
\bar{v}_1^*(d_1\otimes p)\bar{v}_1 & 0\\
0 & 0
\end{bmatrix}
w^*
\in
w(\bar{D}_1\otimes\C)w^* =\bar{D}_2\otimes\C
$
so that 
\begin{equation*}
\bar{v}
\begin{bmatrix}
d_1\otimes p & 0 \\
0 & d_2\otimes p 
\end{bmatrix}
\bar{v}^*
\in 
\bar{v}_2(\bar{D}_2\otimes\C)\bar{v}_2 
=\bar{D}_2\otimes p
=(\bar{D}_0\otimes p)\bar{v}\bar{v}^*.
\end{equation*}
Therefore we have
$
\bar{v} ( \bar{D}_0\otimes p) \bar{v}^*
\subset (\bar{D}_0 \otimes p) \bar{v}\bar{v}^*
$
and similarly
$
\bar{v}^* ( \bar{D}_0\otimes p) \bar{v}
\subset (\bar{D}_0 \otimes p) \bar{v}^*\bar{v}
$
so that we have
$$
\bar{v} ( \bar{D}_0\otimes p) \bar{v}^*
 = (\bar{D}_0 \otimes p) \bar{v}\bar{v}^*
\quad
\text{ and }
\quad
\bar{v}^* ( \bar{D}_0\otimes p) \bar{v}
= (\bar{D}_0 \otimes p) \bar{v}^*\bar{v}.
$$
By the equalities
$$
\bar{v}^*\bar{v}
=
\begin{bmatrix}
1\otimes p & 0 \\
0 & 0 
\end{bmatrix},
\qquad
\bar{v}\bar{v}^*
=
\begin{bmatrix}
0 & 0 \\
0 & 1\otimes p 
\end{bmatrix},
$$
we know that 
$\bar{v}$ commutes with $1\otimes p$
so that we can regard 
$\bar{v}$
as an element of $M(\bar{\A}_0\otimes p) =M(\bar{\A}_0)$.
Thus we obtain a partial isometry $\bar{v}$ in $M(\bar{\A}_0)$
such that  
$$
\bar{v}^*\bar{v}
=
\begin{bmatrix}
1_{\bar{A}_1} & 0 \\
0 & 0 
\end{bmatrix},
\qquad
\bar{v}\bar{v}^*
=
\begin{bmatrix}
0 & 0 \\
0 & 1_{\bar{A}_2} 
\end{bmatrix},
$$
and
$$
\bar{v} \bar{D}_0 \bar{v}^*
 = \bar{D}_0  \bar{v}\bar{v}^*
\quad
\text{ and }
\quad
\bar{v}^* \bar{D}_0 \bar{v}
= \bar{D}_0 \bar{v}^*\bar{v}.
$$
Therefore by Lemma \ref{lem:3.3},
we conclude that $X$ is equivalent to
$X_\theta$
for some isomorphism
$\theta:(\bar{\A}_1,\bar{D}_1)\longrightarrow (\bar{\A}_2,\bar{D}_2).$
\end{proof}
Recall that the subgroups 
$\Aut(\A\otimes\K,\D\otimes\C) $
and
$\Int(\A\otimes\K,\D\otimes\C)$ 
of automorphism group
$\Aut(\A\otimes\K)$
are defined by 
\begin{align*}
\Aut(\A\otimes\K,\D\otimes\C) 
& = \{\beta\in \Aut(\A\otimes\K)\mid \beta(\D\otimes\C) = \D\otimes\C\},\\
\Int(\A\otimes\K,\D\otimes\C) 
& = \{\beta\in \Int(\A\otimes\K)\mid \beta(\D\otimes\C) = \D\otimes\C\}.
\end{align*}
\begin{corollary}
Let $(\A,\D)$ be a relative $\sigma$-unital pair of $C^*$-algebras.
For any relative imprimitivity bimodule
$[X] \in \Pic(\A\otimes\K,\D\otimes\C)$,
there exists an automorphism
$\theta \in \Aut(\A\otimes\K,\D\otimes\C)$
such that $[X] = [X_\theta]$.
Thus we have a exact sequence
\begin{equation*}
1 
\longrightarrow
\Int(\A\otimes\K,\D\otimes\C)
\longrightarrow
\Aut(\A\otimes\K,\D\otimes\C)
\longrightarrow
\Pic(\A\otimes\K,\D\otimes\C)
\longrightarrow
1.
\end{equation*}
\end{corollary}
Let us denote by
$\Out(\A\otimes\K,\D\otimes\C) $
the quotient group
$
\Aut(\A\otimes\K,\D\otimes\C) /\Int(\A\otimes\K,\D\otimes\C). 
$
We then have 
\begin{corollary}
Let $(\A,\D)$ be a relative $\sigma$-unital pair of $C^*$-algebras.
We have
\begin{equation*}
\Pic(\A,\D) =
\Out(\A\otimes\K,\D\otimes\C).
\end{equation*}
\end{corollary}
\begin{proof}
By Lemma \ref{lem:RNEPIC}, we see that
$\Pic(\A,\D) =\Pic(\A\otimes\K,\D\otimes\C)
$
so that we have the desired equality by the preceding corollary.
\end{proof}

\section{Relative Picard groups of Cuntz--Krieger pairs}
In this section, we will study the relative Picard group
$\Pic(\A,\D)$
for the Cuntz--Krieger pairs $(\OA,\DA)$.
By \cite[Lemma 1.1]{KodakaJOT}, for a unitary $u \in M(\SOA)$,
the automorphism $\Ad(u)$ acts trivially on $K_0(\SOA)$. 
We will first show the following proposition
which is a relative version of \cite[Lemma 3.13]{KodakaJOT} 
(Lemma \ref{lem:K1.3} in Appendix).
\begin{proposition}\label{prop:4.1}
Let $\beta \in \Aut(\SOA)$ satisfy
$\beta(\SDA) = \SDA$ and
$\beta_* =\id$ on $K_0(\OA)$.
Then there exists a unitary
$u \in M(\SOA)$
and an automorphism 
$\alpha \in \Aut(\OA)$
such that
\begin{gather*}
\beta = \Ad(u)\circ(\alpha\otimes\id)
\quad\text{ and }
\quad
\alpha_* = \id \text{ on }
K_0(\OA), \\
u(\SDA)u^* = \SDA, \qquad
\alpha(\DA) = \DA.
\end{gather*}
\end{proposition}
To show the above proposition, we provide several lemmas.
\begin{lemma}\label{lem:4.1}
Let $\beta \in \Aut(\SOA)$ satisfy
$\beta(\SDA) = \SDA$ and
$\beta_* =\id$ on $K_0(\OA)$.
Then for each $k \in \N$,
there exists a partial isometry
$w_k \in \SOA$
such that
\begin{gather}
w_k^* w_k = 1\otimes e_{kk}, \qquad
 w_k w_k^* = \beta(1\otimes e_{kk}), \label{eq:4.1.1}\\
w_k(\SDA)w_k^* \subset  \SDA, \qquad
w_k^*(\SDA) w_k \subset  \SDA. \label{eq:4.1.2}
\end{gather}
\end{lemma}
\begin{proof}
Let us denote by $N_s(\SOA,\SDA)$
the normalizer semigroup 
$$
\{ v \in \SOA 
\mid v \text{ is a partial isometry};
v (\SDA )v^* \subset \SDA,\,  v^* (\SDA )v \subset \SDA \}
$$
 of partial isometries in $\SOA$.
Denote by 
$K_0(\SOA,\SDA)$
the Murray-von-Neumann equivalence classes of projections in $\SDA$
by partial isometries in $N_s(\SOA,\SDA)$.
It has been proved in \cite{MaPAMS} that there exists a natural isomorphism
between   
$K_0(\OA)$ and $K_0(\SOA, \SDA). $
Since
$[\beta(1\otimes e_{kk})]
= \beta_*([1\otimes e_{kk}]) = [1\otimes e_{kk}]$,
we have 
$\beta(1\otimes e_{kk})\sim 1\otimes e_{kk}$ 
in
$K_0(\SOA, \SDA). $
We may find
a partial isometry
$w_k \in \SOA$
satisfying the desired conditions.
\end{proof}
\begin{lemma}\label{lem:4.2}
Let $\beta \in \Aut(\SOA)$ satisfy
$\beta(\SDA) = \SDA$ and
$\beta_* =\id$ on $K_0(\OA)$.
Then there exists a unitary
$w \in M(\SOA)$
such that
\begin{gather*}
(\Ad(w^*) \circ \beta)(1\otimes e_{kk}) =1\otimes e_{kk}, \\
(\Ad(w^*) \circ \beta)(\OA\otimes e_{kk}) =\OA\otimes e_{kk}, \\
\Ad(w^*) \circ \beta(\DA\otimes e_{kk}) =\DA\otimes e_{kk}, 
 \end{gather*}
for all $k \in \N$.
\end{lemma}
\begin{proof}
Take $w_k \in  \SOA$ be a partial isometry for each $k \in \N$
satisfying the conditions of Lemma \ref{lem:4.1}.
It is easy to see that the summation 
$\sum_{k=1}^\infty w_k $ converges in the strict topology
of
$M(\SOA)$.
By \eqref{eq:4.1.1} and \eqref{eq:4.1.2}, 
we have
$w^*w= ww^*=1$
and
$w(\SDA)w^*
=w^*(\SDA)w= \SDA.
$
We then see that
\begin{equation*}
w(1\otimes e_{kk})w^* =ww_k^*w_kw^* = w_k w_k^* = \beta(1\otimes e_{kk})
\end{equation*}
so that 
$(\Ad(w^*)\circ\beta)(1\otimes e_{kk}) = 1\otimes e_{kk}.$
For $x \in \OA$, we have
\begin{align*}
(\Ad(w^*)\circ\beta)(x\otimes e_{kk})
=& (\Ad(w^*)\circ\beta)((1\otimes e_{kk})(x\otimes e_{kk})(1\otimes e_{kk})) \\
=&(1\otimes e_{kk})(\Ad(w^*)\circ\beta)((x\otimes e_{kk}))(1\otimes e_{kk})
\end{align*}
so that
$
(\Ad(w^*) \circ \beta)(\OA\otimes e_{kk}) =\OA\otimes e_{kk}.
$
As
$\beta(\SDA) =\SDA$
and
$w^*(\SDA)w =\SDA$, we have
$
\Ad(w^*) \circ \beta(\DA\otimes e_{kk}) =\DA\otimes e_{kk}. 
$
\end{proof} 

{\it Proof of Proposition \ref{prop:4.1}.}
Suppose that
 $\beta \in \Aut(\SOA)$ satisfies
$\beta(\SDA) = \SDA$ and
$\beta_* =\id$ on $K_0(\OA)$.
Take a unitary $w\in M(\SOA)$
satisfying the conditions of Lemma \ref{lem:4.2}.   
Put
$\beta_w =\Ad(w^*)\circ\beta \in \Aut(\SOA).$
Since
$
(\Ad(w^*) \circ \beta)(\OA\otimes e_{kk}) =\OA\otimes e_{kk},
$
we may find an automorphism
$\alpha_k \in \Aut(\OA)$  for $k \in \N$ such that 
\begin{equation*}
\alpha_k(x)\otimes e_{kk} =\beta_w(x \otimes e_{kk})
\quad
\text{ for } 
x \in \OA.
\end{equation*}
By replacing $\beta$ with $\beta_w$, we may assume that 
$\beta_w(x \otimes e_{kk}) =\alpha_k(x)\otimes e_{kk}. 
$
For $j,k \in \N$, we have 
\begin{equation*}
\beta(x \otimes e_{jk})
=\beta((1\otimes e_{jk})(x \otimes e_{jk})) 
=\beta(1\otimes e_{jk}) \cdot (\alpha_k(x) \otimes e_{jk}). 
\end{equation*}
By putting $x =1$,
we see that 
\begin{equation*}
\beta(1 \otimes e_{jk})
=(1\otimes e_{jj})\beta(1 \otimes e_{jk})(1 \otimes e_{kk})
\end{equation*}
so that there exists
$w_{jk} \in \OA$ such that 
$w_{jk}^* = w_{kj}$ and 
$\beta(1 \otimes e_{jk}) = w_{jk}\otimes e_{jk}$.
Since
\begin{equation*}
w_{jk}^*w_{jk} \otimes e_{kk}
={\beta(1\otimes e_{jk})}^* \beta(1 \otimes e_{jk})
=\beta(1\otimes e_{kk}) 
=1 \otimes e_{kk}
\end{equation*}
so that 
$w_{jk}^*w_{jk} = 1$ 
and similarly
$w_{jk}w_{jk}^* = 1$.
We also have for $a \in \DA$ 
\begin{equation*}
w_{jk}^* a w_{jk} \otimes e_{jj}
=\beta((1\otimes e_{jk})(a\otimes e_{kk})(1 \otimes e_{kj})) 
=\beta(a \otimes e_{jj}) =\alpha_j(a) \otimes e_{jj}
\end{equation*}
so that
$w_{jk}\DA w_{jk}^* = \DA$.
Since
\begin{equation*}
\beta(x \otimes e_{jk})
=\beta(1\otimes e_{jk}) \cdot(\alpha_k(x) \otimes e_{kk})
=w_{jk} \alpha_k(x) \otimes e_{jk}
\end{equation*}
and similarly
$\beta(x \otimes e_{jk}) = \alpha_j(x)w_{jk}  \otimes e_{jk},
$
we see
$w_{jk} \alpha_k(x) \otimes e_{jk} =\alpha_j(x)w_{jk}  \otimes e_{jk}$
and hence
$ \alpha_k(x) =w_{jk}^* \alpha_j(x)w_{jk}$ for $x \in \OA$.
Put
$u = \sum_{k=1}^\infty w_{1k}\otimes e_{kk}$
which is easily proved to be a unitary in
$M(\SOA)$.
It then follows that 
\begin{align*}
\beta(x \otimes e_{jk})
=& \beta((1\otimes e_{j1})(x \otimes e_{11})(1 \otimes e_{1k})) \\
=& (w_{j1}\otimes e_{j1})(\alpha_1(x) \otimes e_{11})(w_{1k} \otimes e_{1k}) \\
=& (w_{j1} \alpha_1(x)w_{1k}) \otimes e_{jk} \\
=& u^* (\alpha_1(x) \otimes e_{jk}) u 
\end{align*}
for $x \in \OA$
so that
$\beta = \Ad(u^*) \circ (\alpha_1\otimes \id)$.
Since $w_{1k}\DA w_{1k}^* =\DA$,
we have
$u(\SDA)u^* =\SDA$.
By \cite[Lemma 1.1]{KodakaJOT}, we know that
$\Ad(u)_* =\id$ on $K_0(\OA)$
so that 
$\alpha_{0*} = (\beta^{-1})_* = \id$ on $K_0(\OA)$.
\qed

 We thus have the following theorem.
\begin{theorem}\label{thm:4.1}
Let $\beta \in \Aut(\SOA)$.
Then $\beta$ satsifies the following condition 
\begin{equation}
\beta(\SDA) = \SDA
\quad \text{ and }
\beta_* =\id \text{  on } K_0(\OA)
\end{equation}
if and only if 
 there exists an automorphism 
$\alpha \in \Aut(\OA)$ 
and a unitary
$u \in M(\SOA)$
such that
\begin{gather}
\beta = \Ad(u)\circ(\alpha\otimes\id)
\quad\text{ and }
\quad
\alpha_* = \id \text{ on }
K_0(\OA), \label{eq:thm4.1.1}\\
u(\SDA)u^* = \SDA, \qquad
\alpha(\DA) = \DA.\label{eq:thm4.1.2}
\end{gather}
\end{theorem}
 The following proposition shows that 
 the expression $\beta$ in the form \eqref{eq:thm4.1.1} and \eqref{eq:thm4.1.2}  
 is unique up to inner automorphisms on $\OA$ invariant globally $\DA$.
\begin{proposition}\label{prop:uniquebeta}
Supose that $\beta \in \Aut(\SOA,\SDA)$ is of the form
$\beta = \Ad(u)\circ(\alpha\otimes\id)= \Ad(u')\circ(\alpha'\otimes\id)$
for some  automorphisms 
$\alpha, \alpha' \in \Aut(\OA,\DA)$ 
and unitaries
$u, u' \in M(\SOA)$
satisfying both the conditions
\eqref{eq:thm4.1.1} and \eqref{eq:thm4.1.2}.
Then there exists a unitary $V \in \OA$ 
such that 
\begin{equation}
 u = u'(V\otimes 1) ,\qquad
 \alpha = \Ad(V^*)\circ \alpha'
\quad
\text{ and }
 \quad
 V\DA V^* = \DA.
\end{equation}
\end{proposition}
\begin{proof}
For $x \otimes K\in\SOA$,
we have
$u(\alpha(x)\otimes K)u^* =  u'(\alpha'(x)\otimes K)u'^*$.
Put 
$v =u'^* u \in M(\SOA)$ 
which is a unitary satisfying
$v(\alpha(x)\otimes K) =  (\alpha'(x)\otimes K)v.$
We in particularly see that
$v(1\otimes e_{jk}) =  (1\otimes e_{jk})v$
for all $j, k \in \Zp$.
Define 
$V\in \OA$ by setting
 $V \otimes e_{11} = (1\otimes e_{11})v(1\otimes e_{11})$.
As $v$ commutes with $1\otimes e_{11}$,
we know that
$V$ is a unitary in $\OA$.
We then have
\begin{align*}
u'^* u(1\otimes e_{kk})
& = v(1\otimes e_{k1})(1\otimes e_{1k}) \\
& =(1\otimes e_{k1}) v(1\otimes e_{1k}) \\
& =(1\otimes e_{k1})(V \otimes e_{11})(1\otimes e_{1k}) \\
& =(V \otimes 1)(1\otimes e_{k1})(1\otimes e_{11})(1\otimes e_{1k}) \\
& =(V \otimes 1)(1\otimes e_{kk})
\end{align*}
for all $k \in \Zp$.
 Hence we have
$u'^* u = V \otimes 1$.
As we have for $x\in \OA$
\begin{align*}
\alpha(x) \otimes e_{11}
& =  v^*(\alpha'(x)\otimes e_{11})v\\
& =  v^* (1\otimes e_{11})(\alpha'(x)\otimes e_{11}) (1\otimes e_{11})v\\
& =  (V^* \otimes e_{11})(\alpha'(x)\otimes e_{11}) (V\otimes e_{11})\\
& =  V^* \alpha'(x) V\otimes e_{11}
\end{align*}
so that 
$\alpha(x) = V^* \alpha'(x) V$
for $x\in \OA$.
As $\alpha(\DA) =\alpha'(\DA) = \DA$,
we have
$V \DA V^* =\DA$.
\end{proof}
\begin{corollary}\label{cor:4.1}
Let $\beta \in \Aut(\SOA)$.
Let us denote by  $1_A$ the unit of the $C^*$-algebra $\OA$.
Then $\beta$ satsifies the following condition 
\begin{equation*}
\beta(\SDA) = \SDA
\quad \text{ and }
\beta_*([1_A\otimes e_{11}]) =[1_A\otimes e_{11}] \text{  on } K_0(\SOA)
\end{equation*}
if and only if 
 there exists an automorphism 
$\alpha \in \Aut(\OA)$ 
and a unitary
$u \in M(\SOA)$
such that
\begin{gather*}
\beta = \Ad(u)\circ(\alpha\otimes\id)
\quad\text{ and }
\quad
\alpha_* = \beta_* \text{ on }
K_0(\OA), 
\\
u(\SDA)u^* = \SDA, \qquad
\alpha(\DA) = \DA. 
\end{gather*}
\end{corollary}
\begin{proof}
The if part is clear.
It suffices to show the only if part.
Suppose that
$\beta \in \Aut(\SOA)$
 satsifies the following conditions 
\begin{equation*}
\beta(\SDA) = \SDA
\quad \text{ and }
\quad
\beta_*([1_A\otimes e_{11}]) =[1_A\otimes e_{11}] \text{  on } K_0(\SOA).
\end{equation*}
Since $\beta \in \Aut(\SOA)$ satisfies
$
 \beta_*([1_A\otimes e_{11}]) =[1_A\otimes e_{11}], 
$
By \cite{Ro},
there exists an automorphism
$\alpha_\circ$ of $\OA$ such that 
$\alpha_{\circ*} = \beta_*$ on $K_0(\OA)$.
Hence by using \cite[Proposition 5.1]{MaPAMS}, 
we may find an automorphism
$\alpha_1 $
of $\OA$ such that 
$\alpha_1(\DA) = \DA$
and
$\alpha_{1*} =\alpha_{\circ*}$ on $K_0(\OA)$.
Put
$\beta_1 := \beta\circ (\alpha_1^{-1}\otimes\id) \in \Aut(\SOA)$. 
We have
$\beta_1(\SDA) = \SDA$ and $\beta_{1*} = \beta_* \circ\alpha_{1*}^{-1} =\id$
on $K_0(\OA)$.
By Theorem \ref{thm:4.1},
one may take
an automorphism 
$\alpha_2 \in \Aut(\OA)$ 
and a unitary
$u \in M(\SOA)$
such that
\begin{gather*}
\beta_1 = \Ad(u)\circ(\alpha_2\otimes\id)
\quad\text{ and }
\quad
\alpha_{2*} = \id \text{ on }
K_0(\OA), 
\\
u(\SDA)u^* = \SDA, \qquad
\alpha_2(\DA) = \DA. 
\end{gather*}
Put $\alpha := \alpha_2\circ\alpha_1 \in \Aut(\OA)$.
We then have 
$$
\beta = \Ad(u)\circ(\alpha\otimes\id),
\qquad
\alpha_* = \beta_* \text{ on }
K_0(\OA), 
\qquad
\alpha(\DA) = \DA.
$$
\end{proof}
 Let
\begin{equation*}
\Aut_\circ(\OA,\DA) 
= \{\alpha\in \Aut(\OA)\mid \alpha(\DA) = \DA, \alpha_* = \id \text{ on } K_0(\OA)\}.
\end{equation*}
Since
$\Int(\OA,\DA)$ is a subgroup
of $\Aut_\circ(\OA,\DA)$,
we may consider the quotient group
 $\Aut_\circ(\OA,\DA)/\Int(\OA,\DA)$
which we denote by
$\Out_\circ(\OA,\DA)$.
Thanks to Theorem \ref{thm:4.1}
and Corollary \ref{cor:4.1},
we know the following theorem on the relative Picard group
$\Pic(\OA,\DA)$,
which are relative versions of the results shown in 
Appendix after this section.

Let
$\Psi:\Aut(\OA,\DA)\longrightarrow \Aut(\SOA,\SDA)
$
be the homomorphism
defined by 
$\Psi(\alpha) = \alpha\otimes \id.$
Since
$\Psi(\Int(\OA,\DA)) \subset \Int(\SOA,\SDA),
$
it induces a homomorphism from
$\Out(\OA,\DA)$
to
$\Out(\SOA,\SDA)
$
written 
$\bar{\Psi}$.
The following is a corollary of Proposition \ref{prop:uniquebeta}.
\begin{corollary}\label{cor:injPsi}
The homomorphism
$\bar{\Psi}:\Out(\OA,\DA)\longrightarrow \Out(\SOA,\SDA)
$
is injective.
\end{corollary}
\begin{proof}
Suppose that 
$\alpha\in \Aut(\OA,\DA)$
satisfies
$\alpha\otimes\id = \Ad(u')$ 
for some
$u' \in U(\SOA,\SDA)$.
Put
$\alpha' = \id$ and $u = 1$ in the statement of
Proposition \ref{prop:uniquebeta}
to have a unitary
$V \in U(\OA,\DA)$ such that 
$u' = V\otimes1$ and $\alpha=\Ad(V)$.
\end{proof}
By \cite[Lemma 1.1]{KodakaJOT},
we may define a homomorphism
 $K_* :\Out(\SOA,\SDA) \longrightarrow \Aut(K_0(\SOA))$
by setting
$K_*([\alpha]) = \alpha_*$
for $[\alpha] \in \Out(\SOA,\SDA)$.

\begin{theorem} \label{thm:PicOADA}
Let $A$ be an irreducible non-permutation matrix.
Then the following short exact sequence holds:
\begin{equation}
1\longrightarrow \Out_\circ(\OA,\DA) 
\overset{\bar{\Psi}}{\longrightarrow}\Out(\SOA,\SDA) 
\overset{K_*}{\longrightarrow}\Aut(K_0(\OA\otimes\K))
\longrightarrow 1.  \label{eq:exactOADA1}
\end{equation}
Hence there exists a short exact sequence:
\begin{equation}
1\longrightarrow \Out_\circ(\OA,\DA) 
\overset{\bar{\Psi}}{\longrightarrow}\Pic(\OA,\DA) 
\overset{K_*}{\longrightarrow}\Aut( \Z^N/{(\id - A^t)\Z^N}
\longrightarrow 1.  \label{eq:exactOADA2}
\end{equation}
\end{theorem}
\begin{proof}
We will show the exactness of \eqref{eq:exactOADA1}.
The injectivity of the homomorphism 
$\bar{\Psi}:  \Out_\circ(\OA,\DA)  \longrightarrow \Out(\SOA,\SDA)$
follows from Corollary \ref{cor:injPsi}.
The inclusion relation
$\bar{\Psi}(\Out_\circ(\OA,\DA)) \subset  \Ker(K_*)$ 
is clear.  
 Conversely for any 
$[\beta] \in \Ker(K_*)$,
we know 
$\beta \in \Aut(\SOA)$ satisfy
$\beta_* =\id$ on $K_0(\OA)$.
By Theorem \ref{thm:4.1},
there exist a unitary
$u \in M(\SOA)$
and an automorphism 
$\alpha \in \Aut(\OA,\DA)$
such that
$
\beta = \Ad(u)\circ(\alpha\otimes\id)
$
and
$\alpha_* = \id \text{ on }
K_0(\OA).
$
Hence we have
$
[\beta] = [\alpha\otimes \id] = {\bar{\Psi}}([\alpha])$
and
$[\alpha] \in \Aut_\circ(\OA,\DA)/\Int(\OA,\DA)$,
so that 
we have
$\bar{\Psi}(\Out_\circ(\OA,\DA)) = \Ker(K_*)$

For any $\xi \in \Aut(K_0(\OA\otimes\K))$,
$\xi$ 
gives rise to an automorphism
of the abelian group
$\Z^N/(\id -A^t)\Z^N$.
The group
$\Z^N/(\id -A^t)\Z^N$
is isomorphic to the Bowen--Franks group
$BF(A) =\Z^N/(\id -A)\Z^N$ of the matrix $A$.
By Huang's theorem \cite[Theorem 2.15]{Hu}
and its proof,  any automorphism of 
the Bowen--Franks group
$BF(A)$ comes from an flow equivalence of the
underlying topological Markov shifts
$(\bar{X}_A,\bar{\sigma}_A)$.
It implies that
there exists an automorphism $\psi\in \Aut(\SOA)$
such that 
$\psi(\SDA) = \SDA$ and $\psi_* = \xi$ on $K_0(\OA)$.
Hence $\psi $ belongs to $\Aut(\SOA,\SDA)$ such that
$K_*(\psi) =\xi$.
Consequently 
the sequence 
\eqref{eq:exactOADA1}
is exact.
\end{proof}
Let 
$
\Aut_1( \Z^N/{(\id - A^t)\Z^N})
$ 
be a subgroup of 
$\Aut( \Z^N/{(\id - A^t)\Z^N}) $ defined by
$$
\Aut_1( \Z^N/{(\id - A^t)\Z^N})
= \{ \xi \in \Aut( \Z^N/{(\id - A^t)\Z^N}) \mid \xi([1]) = [1] \}
$$
where 
$[1] \in  \Z^N/{(\id - A^t)\Z^N}$
denotes the class of the vector 
$(1,\dots,1) $ in $\Z^N$. 
\begin{theorem} \label{thm:PicOADA2}
Let $A$ be an irreducible non-permutation matrix.
Then there exists a short exact sequence:
\begin{align*}
1& \longrightarrow \Out(\OA,\DA) 
\overset{\bar{\Psi}}{\longrightarrow}\Pic(\OA,\DA) \\
& \overset{K_*}{\longrightarrow}\Aut( \Z^N/{(\id - A^t)\Z^N}) 
/ \Aut_1( \Z^N/{(\id - A^t)\Z^N})
\longrightarrow 1. 
\end{align*}
\end{theorem}
\begin{proof}
It suffices to show the exactness at the middle.
The inclusion relation
$\bar{\Psi}(\Out(\OA)) \subset  \Ker(K_*)$ is clear.  
Conversely,
by R{\o}rdam's result
\cite{Ro} again,
for any $\xi \in \Aut(K_0(\OA\otimes\K))$
with $\xi([1]) = [1]$,
there exists an automorphism
$\alpha_\circ$ of $\OA$
such that 
$\alpha_{\circ*} = \xi$ on $K_0(\OA)$.
By  \cite[Proposition 5.1]{MaPAMS}, 
we may find an automorphism
$\alpha_1 $
of $\OA$ such that 
$\alpha_1(\DA) = \DA$
and
$\alpha_{1*} =\alpha_{\circ*}$ on $K_0(\OA)$.
Hence
$\alpha_1\in \Aut(\OA,\DA)$
such that
$\bar{\Psi}([\alpha_1]) = \xi$
so that  
$\bar{\Psi}(\Out(\OA)) = \Ker(K_*)$,
and the sequence is exact.
\end{proof}

\section{Appendix:  Picard groups of Cuntz--Krieger algebras}
In this appendix, we will refer to the Picard groups of Cuntz--Krieger algebras
and especially Cuntz algebras.
As examples of the Picard groups for some interesting class of $C^*$-algebras,
  K. Kodaka has studied the Picard groups
for irrational rotation $C^*$-algebras 
$\Pic(A_\theta)$ to show that
$\Pic(A_\theta)$ is isomorphic to
$\Out(A_\theta)$ if $\theta$ is not quadratic,
and a semidirect product
$\Out(A_\theta)\rtimes\Z$
if $\theta$ is quadratic
(\cite{KodakaJLMS}, \cite{KodakaJOT}).
He also studied  the Picard group of certain Cuntz algebras in \cite{KodakaJOT}.
He proved that
$\Pic(\ON) = \Out(\ON)$
for $N=2,3$.
He also  showed  that
there exists a short exact sequence:
\begin{equation}
1\longrightarrow \Out(\ON) 
\overset{\bar{\Psi}}{\longrightarrow}\Pic(\ON) 
\overset{K_*}{\longrightarrow}\Aut( \Z/{(1 - N)\Z})
\longrightarrow 1  \label{eq:exactCuntz}
\end{equation}
for $N=4,6$.
Since $\Aut( \Z/{(1 - N)\Z}) $ is trivial for $N=2,3$,
the Kodaka's results say that the exact sequence \eqref{eq:exactCuntz} holds
for $N=2,3,4,6$. 

We will show that
the above exact sequence holds for all $1<N\in \N$
(Theorem \ref{thm:PicON}).
As a corollary we know that  
the Picard group $\Pic(\ON)$ of the Cuntz algebra $\ON$ is a semidirect product
$\Out(\ON)\rtimes \Z/{(N-2)\Z}$
if $N-1$ is a prime number.

We first refer to the Picard groups of Cuntz--Krieger algebras. 
Let $u \in M(\A)$ 
be a unitary in the multiplier $C^*$-algebra
$M(\A)$ of a $C^*$-algebra $\A$.
The automorphism
$\Ad(u)$ on $\A$
acts trivially on its K-group
$K_0(\A)$
by \cite[Lemma 1.1]{KodakaJOT}.
\begin{lemma}[{Kodaka \cite[Lemma 1.3]{KodakaJOT}}]\label{lem:K1.3}
Let $\beta \in \Aut(\SOA)$ satisfy
$\beta_* =\id$ on $K_0(\OA)$.
Then there exists a unitary
$u \in M(\SOA)$
and an automorphism 
$\alpha \in \Aut(\OA)$
such that
\begin{equation*}
\beta = \Ad(u)\circ(\alpha\otimes\id)
\quad\text{ and }
\quad
\alpha_* = \id \text{ on }
K_0(\OA).
\end{equation*}
\end{lemma}
For a  $C^*$-algebra $\A$,
we put
\begin{equation*}
\Aut_\circ(\A) =\{ \alpha \in \Aut(\A) \mid \alpha_* = \id \text{ on }K_0(\A)\}
\end{equation*}
which is a subgroup of  $\Aut(\A)$.
Since $\Ad(u)_* = \id$ on $K_0(\A)$ for a unitary $u \in M(\A)$,
we see that 
$\Int(\A)$ is a subgroup of 
$\Aut_\circ(\A)$.
The quotient group
$\Aut_\circ(\A) / \Int(\A)$
is denoted by
$\Out_\circ(\A)$.

Let $\A$ be a unital $C^*$-algebra.
Let $\Psi: \Aut(\A) \longrightarrow \Aut(\A\otimes\K)$
be the homomorphism defined by
$\Psi(\alpha) = \alpha\otimes \id$ for 
$\alpha \in \Aut(\A)$.
It induces a homomorphism
$
\bar{\Psi}: \Out(\A) \longrightarrow \Out(\A\otimes\K).
$
If $\bar{\Psi}([\alpha]) =\id$ for some $\alpha \in \Aut(\A)$,
we have 
$\Psi(\alpha) = \Ad(W)$ for some unitary $W \in M(\A\otimes\K)$.
Hence we see that
\begin{equation}
\alpha(x) \otimes K = W(x\otimes K)W^*
\quad \text{ for all } x \in \A, \, K \in \K. 
\end{equation}
Since 
$$
1\otimes e_{11} = \alpha(1)\otimes e_{11} = W(1\otimes e_{11})W^*.
$$
the unitary $W$ commutes $1\otimes e_{11}$
so that there exists a unitary $w \in \A$ such that 
$w \otimes e_{11}= (1\otimes e_{11}) W (1\otimes e_{11})$.
We then have
\begin{equation*}
\alpha(x) \otimes e_{11} = (1\otimes e_{11}) W (x\otimes e_{11})(1\otimes e_{11})
= w x w^* \otimes e_{11} \text{ for all } x \in \A.
\end{equation*}
Hence $\alpha = \Ad(w) \in \Int(\A)$.
This means that 
the map $\bar{\Psi} : \Out(\A) \longrightarrow \Out(\A\otimes\K)$
is injective.
Any automorphism
$\beta \in \Aut(\A\otimes\K)$ 
induces  
an automorphism $\beta_* $ of $K_0(\A \otimes\K)$,
which we denote by $K_*(\beta) \in \Aut(\A \otimes\K)$.
By \cite[Theorem 1.2]{BGR} with
\cite[Corollary 3.5]{BGR}, 
we know $\Pic(\A)=\Pic(\A\otimes\K) = \Out(\A\otimes\K)$. 
\begin{proposition} \label{prop:PicOA}
Let $A$ be an irreducible non-permutation matrix.
Then the following short exact sequence holds:
\begin{equation}
1\longrightarrow \Out_\circ(\A) 
\overset{\bar{\Psi}}{\longrightarrow}\Out(\SOA) 
\overset{K_*}{\longrightarrow}\Aut(K_0(\OA\otimes\K))
\longrightarrow 1. \label{eq:exactOA1}
\end{equation}
Hence there exists a short exact sequence:
\begin{equation}
1\longrightarrow \Out_\circ(\OA) 
\overset{\bar{\Psi}}{\longrightarrow}\Pic(\OA) 
\overset{K_*}{\longrightarrow}\Aut( \Z^N/{(\id - A^t)\Z^N})
\longrightarrow 1.  \label{eq:exactOA2}
\end{equation}
\end{proposition}
\begin{proof}
We will show the exactness of \eqref{eq:exactOA1}.
We have already known that the injectivity of 
$\bar{\Psi}:  \Out_\circ(\OA)  \longrightarrow \Out(\SOA)$.
By definition of the group
$\Aut_\circ(\OA)$, the inclusion relation
$\bar{\Psi}(\Out_\circ(\OA)) \subset  \Ker(K_*)$ is clear.  
 Conversely for any 
$[\beta] \in \Ker(K_*)$,
we know that
$\beta \in \Aut(\SOA)$ satisfy
$\beta_* =\id$ on $K_0(\OA)$.
By Lemma \ref{lem:K1.3},
there exists a unitary
$u \in M(\SOA)$
and an automorphism 
$\alpha \in \Aut(\OA)$
such that
\begin{equation*}
\beta = \Ad(u)\circ(\alpha\otimes\id)
\quad\text{ and }
\quad
\alpha_* = \id \text{ on }
K_0(\OA).
\end{equation*}
Hence we have
$
[\beta] = [\alpha\otimes \id] = {\bar{\Psi}}([\alpha])$
and
$[\alpha] \in \Aut_\circ(\OA)/\Int(\OA)$.
Therefore 
we have
$\bar{\Psi}(\Out_\circ(\OA)) = \Ker(K_*)$

By R{\o}rdam's result
\cite{Ro},
for any $\xi \in \Aut(K_0(\OA\otimes\K))$,
there exists an automorphism
$\beta$ of $\OA\otimes\K$
such that 
$\beta_* = \xi$.
Therefore the map $K_*$ is surjective
to prove the exactness of the sequence
\eqref{eq:exactOA1}.
\end{proof}
Let 
$
\Aut_1( \Z^N/{(\id - A^t)\Z^N})
$ 
be a subgroup of 
$\Aut( \Z^N/{(\id - A^t)\Z^N}) $ defined by
$$
\Aut_1( \Z^N/{(\id - A^t)\Z^N})
= \{ \xi \in \Aut( \Z^N/{(\id - A^t)\Z^N}) \mid \xi([1]) = [1] \}
$$
where 
$[1] \in  \Z^N/{(\id - A^t)\Z^N}$
denotes the class of the vector 
$(1,\dots,1) $ in $\Z^N$. 
\begin{proposition} \label{prop:PicOA2}
Let $A$ be an irreducible non-permutation matrix.
Then there exists a short exact sequence:
\begin{align*}
1& \longrightarrow \Out(\OA) 
\overset{\bar{\Psi}}{\longrightarrow}\Pic(\OA) \\
& \overset{K_*}{\longrightarrow}\Aut( \Z^N/{(\id - A^t)\Z^N}) 
/ \Aut_1( \Z^N/{(\id - A^t)\Z^N})
\longrightarrow 1.
\end{align*}
\end{proposition}
\begin{proof}
It suffices to show the exactness at the middle.
The inclusion relation
$\bar{\Psi}(\Out_\circ(\OA)) \subset  \Ker(K_*)$ is clear.  
Conversely,
by R{\o}rdam's result
\cite{Ro} again,
for any $\xi \in \Aut(K_0(\OA\otimes\K))$
with $\xi([1]) = [1]$,
there exists an automorphism
$\beta$ of $\OA$
such that 
$\beta_* = \xi$.
The sequence is exact.
\end{proof}

We will finally mention about the Picard groups of Cuntz algebras.
By using Proposition \ref{prop:PicOA}, we  know the following theorem.
For $N =2,3,4,6$, Kodaka has already shown in \cite[Corollary 15, Remark 17]{KodakaJLMS}.
\begin{theorem} \label{thm:PicON}
For each $1<N\in \N$,
there exists a short exact sequence:
\begin{equation}
1\longrightarrow \Out(\ON) 
\overset{\bar{\Psi}}{\longrightarrow}\Pic(\ON) 
\overset{K_*}{\longrightarrow}\Aut( \Z/{(1 - N)\Z})
\longrightarrow 1 \label{eq:exactON1}
\end{equation}
\end{theorem}
\begin{proof}
Since $K_0(\ON) =\Z/{(1 - N)\Z}$ by \cite{CuntzAnnMath}
and the unit $1$ of the $C^*$-algebra $\ON$ corresponds to the generator
$[1]$ of the cyclic group $\Z/{(1 - N)\Z}$,
the fact $\alpha(1) =1$ for any automorphism 
$\alpha \in \Aut(\ON)$ ensures us that 
$\alpha_* = \id $ on 
$K_0(\ON) $.
Hence we see that 
$\Aut_\circ(\ON) = \Aut(\ON)$
and hence
$\Out_\circ(\ON) = \Out(\ON)$.
Therefore the exact sequence 
\eqref{eq:exactOA1}
goes to \eqref{eq:exactON1}. 
 \end{proof}

As a corollary, we have
\begin{corollary}\label{cor:PicONprime}
Suppose that 
 $N-1$ is a prime number.
Then the Picard group $\Pic(\ON)$ of the Cuntz algebra $\ON$ is a semidirect product
$\Out(\ON)\rtimes \Z/{(N-2)\Z}$
of the outer automorphism group by the cyclic group
$\Z/{(N-2)\Z}$:
\begin{equation*}
\Pic(\ON)=
\Out(\ON)\rtimes \Z/{(N-2)\Z}.
\end{equation*}
\end{corollary}
\begin{proof}
As $N-1$ is a prime number,
an automorphism $\eta $ of  
the cyclic group
$\Z/(1 - N)\Z$ is determined by
$\eta(1)$ which can take its value in 
$\{1,2,\dots, N-2\}$, so that we have 
$\Aut( \Z/{(1 - N)\Z})$ is isomorphic to
$\Z/{(N-2)\Z}$.
Since $N$ is not prime,
by \cite[Theorem 16]{KodakaJLMS},
for any $k \in \N$ with $1\le k \le N-1$,
there exists $\beta_k \in \Aut(\ON\otimes\K)$
such that $(\beta_k)* = k\cdot \id$ on $K_0(\ON)$.
Hence the correspondence
$k \in \{1,2,\dots, N-1\}\longrightarrow [\beta_k] \in \Pic(\ON)$
gives rise to a cross section for the exact sequence \eqref{eq:exactON1}.
Therefore the exact sequence \eqref{eq:exactON1} splits and yields
a decomposition of $\Pic(\ON)$ into 
a semidirect product
$\Out(\ON)\rtimes \Z/{(N-2)\Z}$.
\end{proof}

\medskip
\begin{remark}
After the first draft of the paper was completed, the following paper has appeared in arXiv.

Kazunori Kodaka, Tamotsu Teruya: The strong Morita equivalence for inclusions of 
$C^*$--algebras and conditional expectations for equivalence bimodules,
 arXiv:1609.08263.
 
 In the above paper, Morita equivalence for pairs of $C^*$-algebras is defined.
 However, their definition of Morita equivalence is different from ours.   
\end{remark}

{\it Acknowledgments:}
This work was  supported by JSPS KAKENHI Grant Number 15K04896.


\end{document}